\documentclass[12pt, reqno]{amsart}
\usepackage{amsfonts}
\usepackage{bbm}
\usepackage{amscd,amsfonts}
\usepackage{amssymb, eucal, amsfonts, amsmath, xypic, latexsym}
\usepackage{pifont}
\usepackage{mathrsfs,color}
\usepackage{amsthm,indentfirst,bm,fancyhdr,dsfont}
\usepackage{graphicx}
\usepackage[all]{xy}
\usepackage[CJKbookmarks=true]{hyperref}

\usepackage{mathrsfs}
\usepackage{amsmath}
\usepackage{amssymb}
\usepackage{hyperref}

\newtheorem{theorem}{Theorem}[section]
\newtheorem{lemma}[theorem]{Lemma}
\newtheorem{prop}[theorem]{Proposition}

\newtheorem{corollary}[theorem]{Corollary}
\theoremstyle{definition}

\newtheorem{defn}[theorem]{Definition}

\newtheorem{remark}[theorem]{Remark}

\newtheorem{hypothesis}{Hypothesis}[theorem]

\def\gl{\mathfrak{gl}(n)}

\def\dim{\text{\rm dim\,}}

\def\bz{{\bar 0}}
\def\bo{{\bar 1}}
\def\id{\textsf{id}}

\def\ggg{\mathfrak{g}}

\def\hhh{\mathfrak{h}}

\def\nnn{\mathfrak{n}}

\def\bbb{\mathfrak{b}}
\def\bbk{\mathbf{k}}
\def\bk{\mathbf{k}}
\def\uk2{U_{\chi}{(\mathfrak{gl}}(2))}

\def\bo{{\bar 1}}
\def\bz{{\bar 0}}
\def\ev{{\text{ev}}}

\newcommand{\bbz}{\mathbb{Z}}
\newcommand{\bbn}{\mathbb{N}}

\newcommand\comment[1]
\numberwithin{equation}{section}

\def\ggg{\mathfrak{g}}

\def\gl{\mathfrak{gl}}

\def\ca{\mathcal{A}}

\def\cz{\mathcal{Z}}

\def\osp{\mathfrak{osp}}

\def\g{\mathfrak{g}}

\def\ggg{\mathfrak{g}}

\def\hhh{\mathfrak{h}}

\def\bbb{\mathfrak{b}}

\def\nnn{\mathfrak{n}}

\def\fw{\mathfrak{W}}

\def\bbf{\mathbb{F}}
\def\bbz{\mathbb{Z}}

\def\bk{\mathbf{k}}
\def\bbn{\mathbb{N}}

\def\bk{{\mathds{k}}}

\def\bg{{\mathbf{g}}}

\def\bo{{\bar 1}}
\def\bz{{\bar 0}}
\def\ev{{\text{ev}}}

\def\Lie{\text{Lie}}
\def\ad{\text{ad}}
\def\ker{\text{Ker}}

\def\id{\mathsf{id}}
\def\Frac{\mbox{Frac}}
\def\rank{\text{rank}}

\def\Spec{\text{Spec}}

{\vskip-\lastskip\medskip
  \noindent
  {\em #1.}\enspace
  }%
{\qed\par\medskip
  }

\def\bk{{\mathbf k}}
\def\bz{{\bar 0}}
\def\bo{{\bar 1}}

\def\ad{\text{ad}}

\def\ev{\text{ev}}

\def\Lie{\text{Lie}}
\def\Frac{\text{Frac}}

\def\Spec{\text{Spec}}

\def\Det{\text{Det}}

\def\id{\text{id}}

\def\gl{\mathfrak{gl}}

\def\osp{\mathfrak{osp}}

\def\mc{\mathcal}
\def\mf{\mathfrak}

\def\bbz{\mathbb Z}

\def\bbf{\mathbb F}

\def\fw{\mathfrak {W}}

\def\cz{\mathcal {Z}}
\def\cd{\mathcal {D}}

\def\g{\textsf{g}}

\def\fs{{\mathfrak s}}

\def\b0{{\bar 0}}
\def\b1{{\bar 1}}

\def\ggg{{\mathfrak g}}
\def\hhh{{\mathfrak h}}
\def\bbb{{\mathfrak b}}
\def\nnn{{\mathfrak n}}

\def\pr1{\text{pr}_1}

\def\fw{\mathfrak{W}}

\def\hc{\text{hc}}

\def\bg{\mathbf{g}}
\def\bG{\mathbf{G}}
\def\bZ{\mathbf{Z}}

\def\bzw{\widetilde{\mathbf{Z}}}
\def\Frac{\text{Frac}}

\def\frakZ{\mathfrak{Z}}
\def\frakzz{\mathfrak{Z}_p}
\def\frakzo{\mathfrak{Z}_\hc}

\def\ad'{{\textsf{ad}'}}
\def\ad{{\textsf{ad}}}
\def\cald{\mathcal{D}}
\def\scrl{\mathscr{L}}
\def\scra{\mathscr{A}}

\begin{document}
\title[Birational equivalence of the Zassenhaus varieties]{Birational equivalence of the Zassenhaus varieties for  basic classical Lie
superalgebras and their purely-even reductive Lie subalgebras in odd characteristic}

{\color{blue}

\author{Bin Shu}
\address{School of Mathematical Sciences, Ministry of Education Key Laboratory of Mathematics and Engineering Applications \& Shanghai Key Laboratory of PMMP,  East China Normal University, No. 500 Dongchuan Rd., Shanghai 200241, China} \email{bshu@math.ecnu.edu.cn}
\author{Lisun zheng}
\address{College of Sciences, Shanghai Institute of Technology, No. 100 Haiquan Rd., Fengxian District, Shanghai 201418, China}
\email{lszhengmath@hotmail.com}
\author{Ye Ren}
\address{School of Mathematical Sciences,   East China Normal University, No. 500 Dongchuan Rd., Shanghai 200241,   China}
\email{523922184@qq.com}
}

 \thanks{The Key words:  basic classical Lie superalgegbras,  centers of universal enveloping algebras, maximal spectrums, Zassenhaus varieties}

\thanks{This work is partially supported by the National Natural Science Foundation of China (12071136 and
12271345), and by Science and Technology Commission of Shanghai Municipality (No. 22DZ2229014).}

\begin{abstract} Let $\ggg=\ggg_\bz\oplus\ggg_\bo$ be a basic classical Lie superalgebra  over an algebraically closed field $\bk$ of characteristic $p>2$.  Denote by $\cz$ the center of the universal enveloping algebra $U(\ggg)$. Then $\cz$ turns out to be finitely-generated purely-even commutative algebra without nonzero divisors.
 In this paper, we demonstrate that the fraction $\Frac(\cz)$ is isomorphic to $\Frac(\frakZ)$ for the center $\frakZ$ of $U(\ggg_\bz)$. Consequently, both Zassenhaus varieties for $\ggg$ and $\ggg_\bz$ are birationally equivalent via a subalgebra $\widetilde\cz\subset\cz$, and $\Spec(\cz)$ is rational   under the standard hypotheses.
\end{abstract}

\maketitle

\section*{Introduction}

The main purpose of the present paper is to develop  the theory of centers of the universal enveloping algebra $U(\ggg)$ for basic classical Lie superalgebras $\ggg=\ggg_\bz\oplus\ggg_\bo$ in odd characteristic,  and to establish the relation with the centers of $U(\ggg_\bz)$.

\subsection{} By  Kac's classification result of finite-dimensional simple Lie superalgebras over complex numbers, a complex finite-dimensional simple Lie superalgebra is either isomorphic to  one from the classical series, or to one from the Cartan-type series (\cite{Kac2}).  Recall that basic classical Lie superalgebras over complex numbers contain the list of simple Lie superalgebras
 $\frak{sl}(m|n)$,  $\frak{osp}(m|n)$;
$\text{F}(4)$,
$\text{G}(3)$,
$\text{D}(2,1,\alpha)$, along with
$\frak{gl}(m|n)$ (note that $\frak{sl}(m|m)$ is not a simple Lie superalgebra, containing a nontrival center).  Let $L=L_\bz \oplus L_\bo$
be a complex basic classical Lie superalgbra with Cartan subalgebra $\hhh$ and Weyl group $\frak{W}$. Denote by $Z(L) = Z(L)_\bz\oplus Z(L)_\bo$ the
center of the enveloping superalgebra $U(L)$, where $Z(L)_i =\{z\in U(L)\mid za = (-1)^{ik}az, \forall a\in L_k \mbox{ for } k\in \bbz_2\}$, $i\in \bbz_2$.

By means of Harish-Chandra homomorphism, the center $Z(L)$ can be imbedded into $S(\hhh)^\fw$, where $S(\hhh)^\fw$ denotes the invariants of the symmetric algebra of $\hhh$ under the action of $\fw$ (cf. \cite[Proposition 2.20]{CW}, \cite{Kac3}). The center is isomorphic to the image of the Harish-Chandra homomorphism, which can be described precisely in the case $\gl$ and $\osp$ (cf. \cite[\S2.25]{CW}). Hence the linkage principle for $\gl$ and $\osp$  can be read off from the center structure (cf. \cite[\S2.26]{CW}, and thereby provides more information for their representation theory.

  Basic classical Lie superalgebras can be well defined over fields of odd characteristic.  The counterparts (by some suitable modification if necessary) of the above series over complex numbers surely can be naturally expected to be the most important series in the classification of simple Lie superalgebras  over an algebraically closed field of odd characteristic (cf. \cite{BKLS},  \cite{ZL}).

   \subsection{} Basic classical Lie superalgebras admit even parts which are reductive Lie algebras. In some sense, the representation theory  of these Lie superalgebras over the complex number field can be regarded as an algebra extension of that of reductive Lie algebras (see for example, \cite{CLW}, \cite{BLW}, {\sl{etc}.}).  The same thing can be expected in modular case (see \cite{BKu}, \cite{SW}, \cite{CSW}, \cite{PS}, \cite{LS}, {\sl{etc}.}). Recall, by Veldkamp's theorem (\cite{Ve}, \cite{KW}, \cite{MR}, \cite{BGo}, {\sl etc.}), the center $\frakZ$ of the universal enveloping algebra $U(\g)$ for a reductive Lie algebra $\bg$ over $\bk$ is generated by the $p$-center $\frakzz$ and the Harish-Chandra center $\frakzo$. Furthermore, the Zassenhaus variety (the spectrum of maximal ideals of $\frakZ$) is rational (see \cite{T}).

   In the present paper, we will develop the theory of centers of universal enveloping algebras for modular basic classical Lie superalgebras.  More precisely, for a basic classical Lie superalgebra $\ggg=\ggg_\bz\oplus\ggg_\bo$ which is the Lie superalgebra of algebraic supergroup $G$, and $\ggg_\bz=\Lie(G_\ev)$  with $G_\ev$ being the purely-even subgroup of $G$ (see \cite{FG1, FG2, Gav}), we first show that the center $\cz=\text{Cent}(U(\ggg))$ is a finitely-generated purely-even commutative algebra (see Proposition \ref{cev}), has no nonzero-divisors (see Lemma \ref{lem: no nonzero divisors}). Furthermore, we can define the $p$-center $\cz_p$ and the Harish-Chandra center $\cz_{\text{hc}}:=Z(U(\ggg))^{G_\ev}$. This $G_\ev$ is a connected reductive algebraic group.  Denote by $\widetilde\cz$ the subalgebra of the center $\cz:=Z(U(\ggg))$ generated by $\cz_p$ and $\cz_{\text{hc}}$.  We first show the structure theorem which says that the fractions $\Frac(\cz)$ and $\Frac(\widetilde\cz)$ are isomorphic (see Theorem \ref{secondtheorem}).

   Next, we further  prove that $\Frac(\widetilde\cz)$ is isomorphic to $\Frac(\frakZ)$ for the center $\frakZ$ of $U(\ggg_\bz)$, and then   $\Spec(\cz)$ is birationally equivalent to $\Spec(\frakZ)$ via $\widetilde\cz$ (see Theorem \ref{thm: two cent birational eq}). Consequently, the Zassenhaus variety $\Spec(\cz)$ is  rational, which means $\Frac(\cz)$ is purely transcendental over $\bk$.

   \subsection{} The present paper is divided into two parts. One is about the center structure of the enveloping algebra of a modular basic classical Lie superalgebra (Section 2).  The other one is devoted to showing the birational equivalence of the Zassehaus varieties for basic classical Lie superalgebras and the purely-even reductive parts, and their rationality (Section 3). The most contents of Section 2 are already accepted for publishing as a part of the survey paper titled ``A survey on centers and blocks for universal enveloping algebras of Lie superalgebras" by the first two authors (B. Shu and L. Zheng) in the 18th ICCM proceedings (\cite{ShuZheng}).

   %The second part is mainly devoted to the bi-rational equivalence and %rationality of the Zassenhaus varieties.

\subsection{} By vector spaces, subalgebras, ideals, modules, and submodules {\sl etc.}, we mean in the super sense unless otherwise specified,  throughout the paper.

\section{Preliminaries: modular basic classical Lie superalgebras}

Throughout the paper, we always assume  $\ggg=\ggg_\bz\oplus\ggg_\bo$ is  a finite-dimensional Lie superalgebra over  an algebraically closed field $\bk$ of characteristic $p>0$. Let $U(\ggg)$ be the universal enveloping superalgebra of $\ggg$. { In general, for a vector superspace $V=V_\bz\oplus V_\bo$ we denote by $|v|$ the parity of a $\bbz_2$-homogeneous element $v\in V$ with $|v|\in \{\bz,\bo\}$. We further express this  by writing  $v\in V_{|v|}$. }

\subsection{Center and anti-center of $U(\ggg)$}\label{sec: 1.1}
Momentarily, we fix an ordered basis $\{x_1,\cdots,x_s\}$ of $\ggg_\bz$ and an ordered basis $\{y_1,\cdots,y_t\}$ of $\ggg_\bo$ for the arguments below.

\subsubsection{} By the PBW theorem,  the enveloping
superalgebra $U(\ggg)$ admits a basis
\begin{align}\label{eq: pbw basis}
x_{1}^{n_{1}}\cdots x_{s}^{n_{s}}y_{1}^{\epsilon_{1}}\cdots y_{t}^{\epsilon_{t}},
 n_i\in\bbn, \; \epsilon_{j}\in\{0,1\}\mbox{ for }i=1,\cdots,s, j=1,\cdots,t.
 \end{align}
Denote by $\cz(\ggg)$ the center of $U(\ggg)$, i.e. $\cz(\ggg):=\{u\in U(\ggg)\mid \ad x(u)=0\;\; \forall x\in \ggg\}$ where $\ad x(u)=xu-(-1)^{|x||u|}ux$ for homogeneous elements $x\in \ggg_{|x|}$ and $u\in U(\ggg)_{|u|}$.
For any $z\in\cz(\ggg)$, $z=z_{0}+z_{1},z_{i}\in U(\ggg)_{\bar
i}, i= 0,1$, we have $z_{i}\in\cz(\ggg)$. This is to say, $\cz(\ggg)$ is a $\mathbb{Z}_{2}$-graded subalgebra of
$U(\ggg)$.
 In the sequel, we will often write $\cz(\ggg)$ simply as $\cz$ provided that the context is clear.

 \subsubsection{} Following Gorelik in \cite{Gor}, we introduce the anti-center $\ca(\ggg)$ of $U(\ggg)$ which is defined below
 $$\ca(\ggg):=\{u\in U(\ggg)\mid \ad'x(u)=0\; \forall x\in \ggg  \}$$
where $\ad'w$ is a linear transform of $U(\ggg)$ by defining  via
$\ad'w(u):=wu- (-1)^{|w|(|u|+1)}uw$ for  homogeneous elements $w\in U(\ggg)_{|w|}$ and  $u\in U(\ggg)_{|u|}$.
%By the same spirit, the center of $U(\ggg)$ can be equivalently %described as
%$$\cz(\ggg)=\{u\in U(\ggg)\mid \ad w(u)=0 \;\forall w\in U(\ggg)\} $$
%with $\ad$ being a linear transform of $U(\ggg)$ defined via
%$\ad u(w):=uw- (-1)^{|u||w|}wu$ for homogeneous elements $u\in %U(\ggg)_{|u|}$ and $w\in U(\ggg)_{|w|}$.
{
 Take an ordered basis $\{y_i\mid i=1,\ldots,t\}$ of $\ggg_\bo$. Then $t$ is an even number when $\ggg$ is a basic classical Lie superalgebra. For any subset $J$ of $\underline t:=\{1,\ldots,t\}$, set $y_J=\prod_{j}y_j$ where  the product is taken with respect to the given order. For the wedge product algebra $\bigwedge^\bullet \ggg_\bo=\sum_{k=0}^t\bigwedge^k \ggg_\bo$,  the top component $\bigwedge^t\ggg_\bo=\bk y_{\underline t}$ is a one-dimensional $\ggg_\bz$-module under adjoint action.

The arguments  in \cite[Lemma 3.1.2]{Gor} and \cite[Theorem 3.3]{Gor} are independent of the characteristic of the base field. This yields  the following observation
\begin{align}\label{eq: equal with anticen}
U(\ggg)=(\ad'U(\ggg))U(\ggg_\bz),
\end{align}
and the modular counterpart of a special case of Gorelik's theorem \cite[Theorem 3.3]{Gor}
\begin{lemma}\label{thm: anti-center nonzero} Let $\ggg=\ggg_\bz\oplus \ggg_\bo$ be a basic classical Lie superlagebra. There exists a nonzero element $v_\emptyset\in U(\ggg)$ such that  the map $\Phi : u\rightarrow (ad' v_\emptyset)u$ provides
a linear isomorphism from the $\ggg_\bz$-invariants of $U(\ggg_\bz)\otimes \bigwedge^t\ggg_\bo$ onto the
anti-center $\ca(\ggg)$.
\end{lemma}

Note that the one-dimensional module $\ggg_\bz$-module  $\bigwedge^t\ggg_\bo$ is a trivial $\ggg_\bz$-module. The above theorem shows that $\ca(\ggg)$ is not zero because it should be isomorphic to the center $U(\ggg_\bz)^{\ggg_\bz}$ of $U(\ggg_\bz)$ as a vector space.

}

\subsection{Restricted Lie superalgebras}\label{restricted}
A Lie superalgebra $\ggg=\ggg_\bz\oplus\ggg_\bo$ is called a restricted Lie superalgebra if $\ggg_\bz$ is a restricted Lie algebra { with $p$-mapping $[p]$,} and $\ggg_\bo$ is a restricted module for $\ggg_\bz$. We still keep the convention that $\ggg$ admits basis $\{x_1,\cdots,x_s\}$ of $\ggg_\bz$ and a basis $\{y_1,\cdots,y_t\}$ of $\ggg_\bo$ for the arguments below.

 %Recall that by the PBW's theorem,  $U(\ggg)$ has basis
%\[x_{1}^{a_{1}}\cdots x_{s}^{a_{s}}y_{1}^{b_{1}}\cdots y_{t}^{b_{t}},
 %a_{i}\in \bbz_{\geq 0}, \; b_{j}\in\{0,1\}\mbox{ for }i=1,\cdots,s, j=1,\cdots,t .\]

By the definition of a restricted Lie superalgebra, the whole $p$-center of $U(\ggg_\bz)$ falls in $\cz$, which is by definition, the subalgebra of $U(\ggg_\bz)$ generated by $x^p-x^{[p]}\in \cz, \forall x\in \ggg_{\bar 0}$. We denote the $p$-center by $\cz_p$.
Set $\xi_{i}=x_{i}^{p}-x_{i}^{[p]}, i=1,\ldots, s$. The $p$-center
$\cz_p$ is a polynomial ring
$\bk[\xi_{1},\ldots,\xi_{s}]$ generated by $\xi_{1},\ldots, \xi_{s}$.

By the PBW theorem, one easily knows that the enveloping
superalgebra $U(\ggg)$ is a free module over $\cz_p$ with basis
\begin{align}\label{eq: free over Zp}
x_{1}^{q_{1}}\cdots x_{s}^{q_{s}}y_{1}^{\epsilon_{1}}\cdots y_{t}^{\epsilon_{t}}, \;
 0\leq q_i\leq p-1, \; \epsilon_{j}\in\{0,1\}\text{ for }i=1,\cdots,s, j=1,\cdots,t.
\end{align}

\subsection{Reduced enveloping algebras of restricted Lie superalgebras}\label{reduced}
Let $V$ be a simple $U(\ggg)$-module for a restricted Lie superalgebra $\ggg$ as the previous subsection. Then Schur's Lemma implies that for each $x\in \ggg_\bz$, $x^p-x^{[p]}$ acts by a scalar $\chi(x)^p$ for some $\chi\in {\ggg_\bz}^*$. We call such a $\chi$ the $p$-character of  $V$. For a given $\chi\in {\ggg_\bz}^*$, let $I_\chi$ be the ideal of $U(\ggg)$ generated by the even central elements $x^p-x^{[p]}-\chi(x)^p$. Generally, a module is called a $\chi$-reduced module for a given $\chi\in \ggg^*$ if  for any $x\in \ggg_\bz$, $x^p-x^{[p]}$ acts by a scalar $\chi(x)^p$. All $\chi$-reduced modules for a given $\chi\in\ggg^*$ constitute a full subcategory of the $U(\ggg)$-module category.
The quotient algebra $U_\chi(\ggg):=U(\ggg)\slash I_\chi$ is called the reduced enveloping superalgebra of $p$-character $\chi$. Obviously, the $\chi$-reduced module category of $\ggg$ coincides with the $U_\chi(\ggg)$-module category.
The superalgebra $U_\chi(\ggg)$ has a basis
\begin{align}\label{eq: pbw basis for chi}
x_{1}^{a_{1}}\cdots x_{s}^{a_{s}}y_{1}^{b_{1}}\cdots y_{t}^{b_{t}},
 0\leq a_{i}\leq p-1; b_{j}\in\{0,1\}\mbox{ for }i=1,\cdots s; j=1,\cdots,t.
 \end{align}
In particular, $\dim U_\chi(\ggg)=p^{\dim \ggg_\bz}2^{\dim \ggg_\bo}$.

When $\chi=0$, the corresponding reduced enveloping algebra is denoted by $U_0(\ggg)$ which is called the restricted enveloping algebra of $\ggg$.

\subsection{Modular basic classical Lie superalgebras and the corresponding algebraic supergroups}\label{restriction}

Let $\ggg=\ggg_\bz\oplus\ggg_\bo$ be a basic classical Lie superalgebra over $\bk$ whose even subalgebras are Lie algebras of reductive algebraic groups, with non-degenerate even supersymmetric bilinear forms.
We  list all the basic classical Lie superalgebras and their even parts over $\bk$ with the restriction on $p$ (cf. \cite{Kac2}, \cite{WZ}. The restriction on $p$ could be relaxed, but we  always assume this  restriction on $p$ throughout the first two sections).
\label{tab: basic cla}
\vskip0.3cm
%\begin{align}
\begin{tabular}
{ccc}
\hline
 Basic classical Lie superalgebra $\frak{g}$ & $\ggg_\bz$  & characteristic of $\bk$\\
\hline
$\frak{gl}(m|n$) &  $\frak{gl}(m)\oplus \frak{gl}(n)$                &$p>2$            \\
$\frak{sl}(m|n)$ &  $\frak{sl}(m)\oplus \frak{sl}(n)\oplus \bk$    & $p>2, p\nmid (m-n)$   \\
$\frak{osp}(m|n)$ & $\frak{so}(m)\oplus \frak{sp}(n)$                  & $p>2$ \\
$\text{F}(4)$            & $\frak{sl}(2)\oplus \frak{so}(7)$                  & $p>15$  \\
$\text{G}(3)$            & $\frak{sl}(2)\oplus \text{G}_2$                    & $p>15$     \\
$\text{D}(2,1,\alpha)$   & $\frak{sl}(2)\oplus \frak{sl}(2)\oplus \frak{sl}(2)$        & $p>3$   \\
\hline
\end{tabular}
%\end{algin}

\vskip0.3cm
For a given Lie superalgebra $\ggg$ in the list, there is an algebraic supergroup $G$  satisfying $\Lie(G)=\ggg$ such that
\begin{itemize}
\item[(1)] $G$ has a subgroup scheme $G_\ev$ which is an ordinary connected reductive group with $\Lie(G_\ev)=\ggg_\bz$;
\item[(2)] There is a well-defined action of $G_\ev$ on $\ggg$, reducing to the adjoint action of $\ggg_\bz$.
    \end{itemize}
\subsubsection{} The above algebraic supergroup can be constructed as a Chevalley supergroup in \cite{FG1}. The pair ($G_\ev, \ggg)$ constructed in this way is called a Chevalley super Harish-Chandra pair (cf. \cite[3.13]{FG2} and \cite[5.5.6]{FG1}). Partial results on $G$ and $G_\ev$ can be found in \cite[Ch. II.2]{Ber}, \cite{Bos2}, \cite{FG1}, \cite[\S3.3]{FG2}, and \cite[Ch.7]{Var}, {\sl etc.}. In the present paper, we will call $G_\ev$ the purely even subgroup of $G$. One easily knows that $\ggg$ is a restricted Lie superalgebra (cf. \cite[Lemma 2.2]{SW} and \cite{SZheng}).

\subsubsection{}\label{sec: froben twist}{
 We denote for a scheme $X$ over $\bk$ by $X^{(1)}$ the Frobenius twist of $X$, which means  the same scheme twisted by the Frobenius morphism, i.e. the structure sheaf is exponentiated to the $p$th power. This notation will be used.}
Recall that $\mbox{Spec}(\cz_p)$ is $G_{\ev}$-equivariantly  isomorphic to
$\ggg_{\bar 0}^{*(1)}$ (cf.  \cite[\S4]{KW}). In the sequel, we will identify $\mbox{Spec}(\cz_p)$ with $\ggg_\bz^{*(1)}$.
Therefore,  we
can consider the coadjoint action of $G_{\ev}$ on $\mbox{Spec}(\cz_p)$. And, the adjoint action of $G_{\ev}$ on $\ggg$  can be extended to the universal enveloping superalgebra $U(\ggg)$, as automorphisms.

\subsection{Root space decomposition}
From now on till the end of \S2, $\ggg$ always denotes a basic classical Lie superalgebra over $\bbk$. Always keep the assumption on  $p$ as listed as in \S\ref{restriction}.

\subsubsection{Basic properties of root spaces}
%Under the restrictions on $p$ as above, we have the following
%proposition for all basic classical Lie superalgebras:
\begin{prop} (cf. \cite[2.5.3]{Kac2} and \cite[\S2]{WZ}) \label{pro1} Let $\ggg$ be a basic classical Lie superalgebra over $\bk$. Then $\ggg$ has its root space decomposition with respect to a Cartan subalgebra $\hhh$:  $\ggg=\mathfrak{h}\oplus \bigoplus\limits_{\alpha\in\Delta}\ggg_{\alpha}$, which satisfies the following properties:
\newcounter{forlists}
\begin{itemize}
\item[(1)]  Each root space $\ggg_\alpha$ is one-dimensional.
\item[(2)] For $\alpha,\beta\in \Delta$, $[\ggg_{\alpha},\ggg_{\beta}]\ne 0$ ~~if and only if $\alpha+\beta\in \Delta$.
%\item[(3)] $\dim\ggg_{\alpha}=1$ $\forall \alpha\in \Delta;$
\item[(3)] There is a non-degenerate super-symmetric invariant bilinear form $(\;,\;)$ on $\ggg$. And $(\ggg_{\alpha},\ggg_{\beta})=0$ for
$\alpha\neq -\beta;$ the form $(\;,\;)$ determines a non-degenerate pairing
of $\ggg_{\alpha}$ and $\ggg_{-\alpha}$ and the
restriction of $(\;,\;)$ on $\mathfrak{h}$ is non-degenerate.
\item[(4)] $[e_{\alpha},e_{-\alpha}]=(e_{\alpha},e_{-\alpha})h_{\alpha}$,
where $h_{\alpha}$ is a nonzero vector determined by
$(h_{\alpha},h)=\alpha(h)$,$\forall~~h\in\mathfrak{h}$.
\end{itemize}
\end{prop}

\subsubsection{}\label{purelyevenarguments}
From now on, we will always assume that $G$ is a basic classical algebraic supergroup as described in \S\ref{restriction}, and assume  $\ggg=\text{Lie}(G)$. Then the purely even subgroup
$G_{\ev}$ is a reductive algebraic group,  and  $\ggg=\ggg_\bz\oplus\ggg_\bo$ is naturally a basic classical Lie superalgebra with natural restricted structure which arises from  $\ggg_\bz=\Lie(G_\ev)$. Associated with a given Cartan subalgebra $\hhh$,

{We make a choice of positive roots, giving a decomposition  $\Delta=\Delta^+\cup\Delta^-$.}
One has root space decompositions for $\ggg_\bz$ and for $\ggg$ respectively: $$\ggg_\bz=\hhh\oplus \bigoplus_{\alpha\in \Delta_0^+}\ggg_\alpha\oplus\bigoplus_{\alpha\in\Delta_0^-}\ggg_\alpha$$ and
$$\ggg=\hhh\oplus \bigoplus_{\alpha\in \Delta^+}\ggg_\alpha\oplus\bigoplus_{\alpha\in\Delta^-}\ggg_\alpha.$$
Recall that $\Delta^+=\Delta_0^+\cup\Delta_1^+$, $\Delta^-=\Delta_0^-\cup\Delta_1^-$, and $\Delta=\Delta^+\cup\Delta^-$.
One has a standard Borel subalgebra $\bbb$ of $\ggg$ with $\bbb=\hhh\oplus\nnn^+$ for $\nnn^+=\bigoplus_{\alpha\in \Delta^+}\ggg_\alpha$.
Set $\rho={1\over 2}(\sum_{\alpha\in\Delta_0^+}\alpha-\sum_{\beta\in \Delta_1^+}\beta)$, and set $r=\dim\hhh$. {In particular, $\hhh$ admits a toral basis $H_i$ ($i=1,\ldots,r$) satisfying $H_i^{[p]}=H_i$ (see for example, \cite[Theorem 2.3.6(2)]{StF}).} Note that there is a non-degenerate even supersymmetric bilinear form on $\ggg$. Especially, the restriction of this non-degenerate form to $\hhh$ is non-degenerate. Hence we can define the non-degenerate form on $\hhh^*$, which is still denoted by $(\cdot,\cdot)$. For $\nu\in\hhh^*$ we can uniquely define $h_\nu\in \hhh$  such that $\lambda(h_\nu)=(\lambda,\nu)$ for any $\lambda\in \hhh^*$.

%\section{The center is an integral domain}

   % From now on,  we will maintain the notations and assumptions as in Theorem %\ref{secondtheorem}. In particular, $\ggg$ will always be assumed to be a %basic classical Lie superalgebra over $\bk$,  under the condition as lited in %\S\ref{restriction}.

\subsection{All central elements in $U(\ggg)$ are even} \label{CentralPrim}

We make use of the following property of the anti-center $\ca(\ggg)$.

\begin{lemma}\label{lem: anti-cent} Let $\ggg$ be a basic classical Lie superalgebra over $\bk$.  All elements of $\ca(\ggg)$ are even. \end{lemma}

\begin{proof} This is a consequence of  the modular counterpart of \cite[Corollary 3.1.3]{Gor}. We make an account for the reader's convenience. By Lemma \ref{thm: anti-center nonzero} and its remark,
it is already known that $\ca(\ggg)\ne0$.

Retain the notations in \S\ref{sec: 1.1}. In particular, $\ggg_\bo$ admits an ordered basis $\{y_i\mid i=1,\ldots,t\}$. Then $t$ is an even number in our case (when $\ggg$ is of basic classical Lie superalgebras). For any subset $J$ of $\underline t:=\{1,\ldots,t\}$, set $y_J=\prod_{j}y_j$ where  the product is taken with respect to the given order.

 Take a nonzero $u\in \ca(\ggg)$. By the observation (\ref{eq: equal with anticen}),  $u$  can be written in the form $u = \sum_{J\subset\underline t}(\ad' y_J)u_J$ where $u_J \in U(\ggg_\bz)$. We can further write $u=\sum_i u_i$ with $u_i=\sum_{\underset{\#J=i}{J\subset\underline t}}(\ad' y_J)u_J$  where $\#J$ denotes the number of the elements of $J$. Then those $u_i$ are linearly independent by the PBW theorem. Moreover, we set  $m_u:=\max\{\#J\mid u_J\ne 0\}$.

We claim that $m_u$ must be $t$. We show this by contradiction. Suppose $m_u<t$. Take $J$ such that $\# J=m_u$ and $u_J\ne 0$. There exists  $i\in \underline t$ which does not belong to $J$. Modulo  $\sum_{\underset{\#I<m_u+1}{I\subset\underline t}}(\ad' y_I)U(\ggg_\bz)$ we have
$$\ad' y_i(u)=\sum_{\underset{\#I=m_u}{I\subset\underline t}}\ad' y_iy_I(u_I)\ne 0$$
because $u_J\ne 0$ and $i\in \underline t\backslash J$, consequently $y_iu_J\ne 0$, and then the PBW basis (\ref{eq: pbw basis}) entails the above nonzero. This  contradicts the assumption that $u$ is a nonzero element in $\ca(\ggg)$. The claim is proved.

Note that $t$ is even. The arbitrariness of such $u$ entails the desired result.  The proof is completed.
\end{proof}

\subsubsection{Harish-Chandra projections}
%\begin{remark}
\label{sec: harish-chandra}
%Let $\fw$ be the Weyl group of $\ggg_\bz$. Then the bilinear form in %Proposition \ref{pro1} is $\fw$-invariant. Define the translated action %of $\fw$ on $\hhh*$ by the formula $w.\lambda=w(\lambda+\rho)-\rho$ for %$w\in\fw$ and $\lambda\in \hhh^*$. Define the left translated action of %$\fw$ on the symmetric algebra $\mf S(\hhh)$  by setting $w.f(\lambda) = %f(w^{-1}\lambda)$ for
%any $f\in \hhh^*$ (we identify $\mf S(\hhh)$ with the coordinate ring %$\bk[\hhh^*]$).
%Note that there is a non-degenerate even supersymmetric bilinear form on %$\ggg$. Especially, the restriction of this non-degenerate form to %$\hhh$ is non-degenerate. Hence we can define the non-degenerate form on %$\hhh^*$, which is still denoted by $(\cdot,\cdot)$. For $\nu\in\hhh^*$ %we can uniquely define $h_\nu\in \hhh$  such that %$\lambda(h_\nu)=(\lambda,\nu)$ for any $\lambda\in \hhh^*$. Set
%$$\hbar:=\prod_{\alpha\in \Delta^+_1}(h_\alpha+(\rho, \alpha))\in %\frak{S}(\hhh).$$
%Recall that  %$\fw$ naturally acts on $\hhh$ and $\hhh^*$ respectively.

Recall that the Harish-Chandra homomorphism
\begin{align}\label{eq: hc general}
\gamma:U(\ggg)\rightarrow U(\hhh)
\end{align}
 is by definition the composite of the canonical projection $\gamma_1: U(\ggg)\rightarrow U(\hhh)$ and the algebra homomorphism $\beta:U(\hhh)\rightarrow U(\hhh)$ defined via $h\mapsto h+\rho(h)$ for all $h\in \hhh$.

{
\subsubsection{Baby Verma modules}\label{Vermamodules} Associated with $\chi\in \hhh_\bz^*\subset \ggg_{\bz}^*$  and $\lambda\in \Lambda(\chi):=\{\lambda\in \hhh^*\mid \lambda(h)^p-\lambda(h^{[p]})=\chi(h)^p,\forall\,h\in\hhh\}$,
 one has a so-called baby Verma module $Z_\chi(\lambda)$ which is by definition an induced $U_\chi(\ggg)$-module arising from a one-dimensional module $\bk_\lambda$ of $\mathfrak{b}$ presented by
$(h+n)v_{\lambda}=\lambda(h)v_{\lambda}$, for $\forall
h\in\mathfrak{h}, n\in\mathfrak{n}^{+}$, where $v_\lambda$ is a basis vector of one-dimensional space $\bk_\lambda$.
By the structure of $Z_\chi(\lambda)$, $Z_\chi(\lambda)=U_{\chi}(\ggg)\otimes_{U_{\chi}(\mathfrak{b})}\bk_{\lambda}$ which coincides with $U_{\chi}(\nnn^-)\otimes v_\lambda$, as $\bk$-spaces, where $\nnn^-=\bigoplus_{\alpha\in \Delta^-}\ggg_\alpha$.
 So we have $\dim
Z_\chi(\lambda)=p^{(\dim\ggg_{\bar
0}-r)/2}2^{\dim\ggg_{\bar 1}/2}$.  By Wang-Zhao's result, $Z_\chi(\lambda)$
is an irreducible $U_{\chi}(\ggg)$-module when
 $\chi$ is  regular semisimple, i.e. $\chi(h_\alpha)\ne0$ for all $\alpha\in \Delta^+$. We denote by $\Omega$ the set of regular semisimple elements in $\hhh^*$.
  Parallel to \cite[Lemma 3.1.4]{Musson} for the base field of characteristic $0$, we have the following observation.

 \begin{lemma}\label{lem: ann} $\bigcap_{\chi\in\Omega}\bigcap_{\lambda\in \Lambda(\chi)}\textsf{Ann}_{U(\ggg)}(Z_\chi(\lambda))=0$.
 \end{lemma}

%The proof will be presented after some necessary preparation.
\begin{proof}
In the following we denote by $\mathscr{A}$ the left hand side of the quality in the lemma. Note that the set of $p^r$ modules $\{Z_\chi(\lambda)\mid \lambda\in \Lambda(\chi)\}$ constitute the complete one of iso-classes of irreducible $U_\chi(\ggg)$-modules under such a circumstance (cf. \cite[Corollary 5.7]{WZ}). Moreover, $U_\chi(\ggg)$ is semisimple for $\chi\in \Omega$ ({\sl{ibid.}}).
Hence, for any given $a\in\scra$, $aU_\chi(\ggg)=0$ for all $\chi\in \Omega$. Consequently, $a\in I_\chi$ (see \S\ref{reduced}). Furthermore,  $a\in \bigcap_{\chi\in \Omega}I_\chi$ because of the arbitrariness of $\chi$ in $\Omega$.
 Note that $\bigcap_{\chi\in \Omega}I^0_\chi$=0 by a classical result (see for example,  the consequence of \cite[Corollary 2.2]{Yao})
where $I^0_\chi$ denotes the ideal of $U(\ggg_\bz)$ generated by $x^p-x^{[p]}-\chi(x)^p$ for all $x\in \ggg_\bz$.  However  $$\bigcap_{\chi\in \Omega}I_\chi\cong \bigcap_{\chi\in \Omega}I^0_\chi\otimes  \bigwedge^\bullet\ggg_\bo$$
by the PBW theorem for $U_\chi(\ggg)$ (see (\ref{eq: pbw basis for chi})), which  means $\bigcap_{\chi\in \Omega}I_\chi=0$.
This completes the proof.
%
%we can uniquely write it as
%$$a=a^0+a^++ \sum_i a_i^{-}a_i^0$$
%with $a^0, a_i^0\in U(\hhh)$, and $a^+\in %U(\nnn^-)U(\hhh)U(\nnn^+)\nnn^+$, and $a_i^-\in U(\nnn^-)\nnn^-$ with %all $a_i$ being linearly independent. From $av_\lambda=0$ it is easily %deduced that
%$$\lambda(a^0)=0=\lambda(a_i^0)$$
%for all $\lambda\in \bigcup_{\chi\in\hhh^*}\Lambda(\chi)$.
\end{proof}

 \iffalse
 In our case ($\ggg$ is a basic classical Lie superalgebra),  we can define Verma module for $\lambda\in\hhh^*$ as
$$\widetilde M(\lambda)=U(\ggg)\otimes_{U(\bbb)}\bbf_\lambda$$
where $\bbf_\lambda$ is a one-dimensional $U(\bbb)$-module with trivial $\nnn^+$-action.
%
 Parallel to \cite[Lemma 13.1.4]{Musson} in the case of complex numbers, we have the following
\begin{align}\label{eq: trival ann}
\bigcap_{\lambda\in \hhh^*}\textsf{Ann}_{U(\ggg)}(\widetilde M(\lambda))=0.
\end{align}
\fi
%Parallel to \cite[Corollary 4.2.4]{Gor}, one can show the following %statement in our case:
%\begin{itemize}
% \begin{lemma}
% The restriction of the Harish-Chandra projection $\gamma$ provides a %linear bijective map  $\ca(\ggg)\rightarrow \hbar %\frak{S}(\hhh)^{{\fw}.}$.
%\end{lemma}

%\begin{proof}
%\end{proof}

%Furthermore, we have the following result.
\subsubsection{} We have the following consequence of Lemma \ref{lem: ann}, which is the modular counterpart of Gorelik's result \cite[Lemma 4.2.5]{Gor}.
\begin{corollary}\label{cor: anticent nonzerodiv} Any nonzero element $a\in \ca(\ggg)$ is not a zero-divisor in $U(\ggg)$.
\end{corollary}

\begin{proof} Consider $Z_\chi(\lambda)$ with $\chi\in \Omega$ and $\lambda\in \Lambda(\chi)$, we might as well suppose that $v_\lambda$ has even parity, without loss of any generality.  Then $Z_\chi(\lambda)=Z_\chi(\lambda)_\bz \bigoplus Z_\chi(\lambda)_\bo$ where $Z_\chi(\lambda)_\bz$ is exactly $U_\chi(\nnn^-)_\bz\otimes v_\lambda$ and $Z_\chi(\lambda)_\bo$ coincides with  $U_\chi(\nnn^-)_\bo\otimes v_\lambda$. By the definition of $\ca(\ggg)$ along with Lemma \ref{lem: anti-cent}, making use of the Harish-Chandra projection defined in (\ref{eq: hc general}) we have
%we have that $a$ acts by multiplication
%by $P(a)(\lambda)$ (resp. $―P(a)(\lambda)$) on the even (resp. odd) $\bbz_2$-homogeneous part of $Z_\chi(\lambda)$.
$$aZ_\chi(\lambda)=\lambda(\gamma(a))
(Z_\chi(\lambda)_\bz-Z_\chi(\lambda)_\bo)$$
where $\lambda$ is regarded as a function on $\mathcal{S}(\hhh)$ in the natural way, which means, identifying $\gamma(a)$ with a polynomial $P(a)\in \bbk[\hhh^*]$ (hence we have $\lambda(\gamma(a))=P(a)(\lambda)$).
 By Lemma \ref{lem: ann},  $P(a)$ is never zero. This means $\gamma(a)\ne 0$. Hence $\gamma|_{\ca(\ggg)}$ is injective.

 Assume that $ua=0$ for some $u\in U(\ggg)$. Then $ua. Z_\chi(\lambda)=0$ for any $\chi\in \Omega$ and $\lambda\in \Lambda(\chi)$.
  Hence $u$ annihilates $Z_\chi(\lambda)$ as long as $P(a)(\lambda)\ne 0$. Note that $P(a)\ne 0$. Then $\bigsqcup_{\chi\in\Omega} \{\lambda \in\Lambda(\chi)\mid P(a)(\lambda)\ne0 \}=(\bigsqcup_{\chi\in\Omega} \{\lambda \in\Lambda(\chi)\})\cap \{\lambda\in\hhh^*\mid P(a)(\lambda)\ne0\}$ is still a dense subset of $\hhh^*$. The same arguments as in the proof of Lemma \ref{lem: ann} yields that $u=0$. This means that $a$ is not a right zero divisor, which further implies, by the definition of $\ca(\ggg)$, that $a$ is not a left zero divisor. The proof is completed.
\end{proof}

}

%Recall that $\mathcal {Z}(\ggg)=\mathcal {Z}(\ggg)_\bz\oplus \mathcal %{Z}(\ggg)_\bo$ with
%$\mathcal {Z}(\ggg)_d:=\{ z\in U(\ggg)_i\mid za=(-1)^{id}az, \forall %a\in \ggg_i \mbox{ for } i\in \bbz_2\}$, $d\in\bbz_2$.

%\renewcommand{\labelenumi}{(\roman{enumi})}

\subsubsection{}
\begin{prop}\label{cev} Let $\ggg$ be a basic classical Lie superalgebra. Then the following statements hold.
\begin{enumerate}
\item[(1)] Each element in $\cz(\ggg)$ has even degree, that
is to say, $\cz(\ggg)\subset U(\ggg)_{\bar 0}$.
\item[(2)]
$\cz(\ggg)$ is a finitely generated $\cz_p$-module,  integral over $\cz_p$. In particular, $\cz(\ggg)$ is a finitely generated commutative algebras over $\bk$.
\item[(3)] Any nonzero elements of $\cz(\ggg)$ are nonzero divisors.
\end{enumerate}
\end{prop}

\begin{proof} For the parts (1) and (3), by Lemma \ref{lem: anti-cent},  $\mathcal{A}(\ggg)\subset
U(\ggg)_\bz$. On the other hand, for  homogeneous nonzero elements $z\in \cz(\ggg)$ and $a\in \mathcal{A}(\ggg)$, and for a homogeneous element $x\in
\ggg$, one has $x(za)=(-1)^{|z||x|}zxa=(-1)^{|z||
x|}(-1)^{|x|(|a|+1)}zax=(-1)^{|x|(|z|+|a|+1)}(za)x$.
Therefore, $za\in\mathcal {A}(\ggg)$, so $|z|+|a|=\bar 0$ by Lemma \ref{lem: anti-cent} again. {Due to Corollary \ref{cor: anticent nonzerodiv}, $za\ne 0$. This
implies that $z$ is an even element because $a\in \ca(\ggg)$ is even. In the same time, by Corollary \ref{cor: anticent nonzerodiv} again, $z$ is nonzero divisor. The part (3) is proved.
}

For the part (2), we know that $U(\ggg)$ is a free $\cz_p$-module with finite
basis. Since $\cz_p$ is isomorphic to a polynomial ring,
$U(\ggg)$ is a Noetherian $\cz_p$-module. Thus $\cz$
is finitely generated $\cz_p$-submodule of $U(\ggg)$.
Assume $x_{1},\cdots, x_{l}$ are a set of generators of $\cz$ over $\cz_p$.
Then for each $z\in \cz$, we have  equations with coefficients in $\cz_p$:  $$zx_{i}=\sum a_{ij}x_{j},
i=1,\cdots, l. $$
We have
$\sum(z\delta_{ij}-a_{ij})x_{j}=0,i=1,\cdots,l$. Therefore, we have
\begin{equation}\label{int}
(zI-A)\left(\begin{array}{c}
            x_{1}\\
            \vdots\\
             x_{l}
             \end{array}\right)
             =0,
\end{equation}
where $I$ is the identity matrix, and $A=(a_{ij})_{l\times  l}$. Multiplying the adjoint matrix of
the coefficient matrix on the two sides of the equation (\ref{int}),
we can obtain
%在方程~(\ref{int})~两边乘上系数矩阵的伴随矩阵, 则可得
%\begin{equation}
\[
\text{det}(zI-A)\left(\begin{array}{c}
            x_{1}\\
            \vdots\\
             x_{l}
             \end{array}\right)
             =0.
\]
%\end{equation}
This implies  that $\text{det}(zI-A)=0$. Hence $\cz(\ggg)$ is
integral over $\cz_p$.
\end{proof}

%Due to Proposition \ref{pro1}, the arguments in \cite[Theorem 15.4.2]{Musson} %can be adopted here.
%So we know that $U(\ggg)$ is prime, which implies that $\cz$ has no zero-divisor in
%$U(\ggg)$， due to J. Wei's observation.
%\footnote{We thank J. Wei for her observation on this point.}
By Proposition \ref{cev}, $\cz$ is a domain. Denote by $\bbf$ the fraction field $\text{Frac}(\cz)$ of $\cz$.
Set
\begin{align}\label{eq: fractional alg}
\cd(\ggg)=U(\ggg)\otimes_{\cz}\bbf
\end{align}
which is a fraction ring of $U(\ggg)$ over $\bbf$.

\section{Centers of universal enveloping algebras of modular basic classical Lie superalgebras}

 Keep the notations as before. In particular,  let $\ggg=\ggg_\bz\oplus \ggg_\bo$ be  a basic classical Lie superalgebra over an algebraically closed   $\bk$ of characteristic $p>0$ as listed in \S\ref{restriction}.
%The center $\mathcal {Z}(\ggg)$ of the universal enveloping algebra $U(\ggg)$ turns out to be commutative instead of being supercommutative (cf. \cite[\S2.2.2]{CW}) (see \S\ref{restriction}).  Furthermore, from \cite[Theorem 15.4.2]{Musson} it follows that this center $\cz(\ggg)$ is domain under our some mild condition.

\subsection{}
\vskip0.15cm
\begin{theorem} (Structure Theorem)\label{secondtheorem}  Let $\widetilde{\mathcal{Z}}$
be the subalgebra of $\mathcal{Z}$ generated by ${\cz}_{p}$
and ${\cz}^{G_{\ev}}$, and let $\hhh$  be the Cartan subalgebra of $\ggg_\bz$ and  $\fw$ be the Weyl group of $\ggg_\bz$. Then
\begin{itemize}
\item[(1)] $\Frac(\mathcal {Z}^{G_{\ev}})$ is isomorphic to $\Frac(U(\frak h)^{\fw})$ as fields.
\item[(2)] $\Frac(\cz)=\Frac(\widetilde\cz)$.
\end{itemize}
\end{theorem}

\begin{remark}  (1) Recall in the ordinary Lie algebra situation, for a connected semi-simple and simply-connected algebraic group $G_0$ over an algebraically closed field $\bk$ of characteristic $p>0$,    $Z=Z(\ggg_0)$, Veldkamp's theorem shows that the center of the universal enveloping algebra $U(\ggg_0)$ for $\ggg_0=\Lie(G_0)$ is generated by its $p$-center $Z_p$ arising from the Frobenius twist ${\ggg^*}^{(1)}$ and its Harish-Chandra center $U(\ggg_0)^{G_0}$ if $p$ is very good (cf. \cite{Ve}, \cite{KW}). Furthermore, $Z$ is a free $Z_p$-module of rank $r$ where $r:=\rank(\ggg_0)$ under the same condition (cf. \cite{Ve}, \cite{KW}.)

Theorem \ref{secondtheorem} provides a super version of Veldkamp's theorem.

{(2) Generally, $\widetilde\cz$ is a proper subalgebra of $\cz$ (one can expect $\widetilde\cz=\cz$ only when $\ggg=\mathfrak{osp}(1,2n)$). The reason is similar to the case of basic classical Lie algebras over the complex number field,   where the Harish-Chandra homomorphism is injective, generally not surjective (see for example,  \cite[\S13.2]{Musson}).}
 \end{remark}

\subsubsection{} We can analyse something behind Theorem \ref{secondtheorem},  by comparing this result with that for the ordinary Lie algebra case.  We recall the geometric
features of the center of the universal enveloping algebra of a
reductive Lie algebra in prime characteristic. Denote by $\mbox{Spec}(R)$
the spectrum of maximal ideals of $R$ for a finitely-generated integral domain $R$ over $\bk$ (such a spectrum is often called a maximal spectrum). Then, for a reductive Lie algebra $\ggg$, the Zassenhaus variety $\textsf{X}=\mbox{Spec}(\cz(\ggg))$ is a normal variety (cf. \cite{Zass}), the locus of  smooth points in this
variety coincides with the Azumaya locus of $U(\ggg)$ which reflects irreducible representations of maximal dimension for $\ggg$ under some mild condition (cf. \cite{BGo})\footnote{The same story happens on modular finite $W$-algebras, see \cite{SZeng}}. Such an algebraic-geometric feature connects the
key information of representations of $\ggg$  to the geometry of $\textsf{X}$ (cf.
\cite{BGo, KW, MR, Ve}).
When turning to  the super case,  one finds that the situation becomes much more complicated  because of appearance of odd part and zero-divisors, and because of the absence of
normality of the Zassenhaus variety. It is a great challenge to make some clear and general investigation on the connection between modular representations of classical Lie superalgebras and the geometry of the corresponding Zassenhaus varieties. The study of the former was initiated by Wang and Zhao (cf. \cite{WZ} and \cite{WZ2}). Anyway, it is the first step to the above question to study the centers of the enveloping superalgebras.

\begin{remark}
It is necessary to remind the readers of comparing the center structure of enveloping algebras of Lie superalgebras and Harish-Chandra homomorphism in the modular case (as in the present paper) with that in the complex number case. For the latter, one can refer to \cite{CW} and \cite{Kac3} and \cite[\S13]{Musson} {\sl{etc.}}.
\end{remark}

\subsection{Restriction homomorphisms}%\label{sec: beginning proof}

The remaining part of this section  will aim at  the proof Theorem \ref{secondtheorem}. The spirit is to  exploit the arguments in  \cite{KW}  for the ordinary Lie algebra case (resp.  for the case of quantum groups at unity root \cite{CKP}) to the supercase.

\subsubsection{} \label{Weylgroups} Recall that $\ggg_\bz$ admits the Weyl group $\fw$ which  naturally acts on $\hhh$ and $\hhh^*$ respectively. Furthermore, note that $p>2$, this action is faithful (cf. \cite[2.3]{KW}). This is to say $\textbf{Z}_\fw(X)=1$ for $X=\hhh$ or $\hhh^*$. Here, $\textbf{Z}_\fw(X)$ is denoted the subgroup of pointwise stablizers of $X$, which is a normal one of $\fw$.
Recall we already have $\Omega=\{\chi\in \hhh^*\subset \ggg^{*}\mid
\chi(h_{\alpha})\neq 0, \alpha\in \Delta^+\}$, the set of  regular
semisimple elements over $\hhh^*$. {Denote by $\Omega^{(1)}$ the Frobenius twist of $\Omega$ in the same meaning  as in \S\ref{sec: froben twist}.
Set $\Omega_{1}:=\{\chi\in \Omega\mid {w}(\chi)\ne\chi, \forall\,
{w(\ne \id)}\in {\fw}\}$.
 Observe  that $G_\ev\cdot \Omega_1$ forms a $G_\ev$-stable open subset of ${\ggg_\bz}^*$ (see \cite[\S2.15]{Steinb})}. Elements of $G_\ev\cdot\Omega_{1}$ are
called strongly regular semisimple ones. Furthermore, denote $\fw_\chi:=\{w\in \fw\mid w(\chi)=\chi\}$, the stabilizer of $\chi$ in $\fw$.
By the definition of $\Omega_1$,
%\begin{align}\label{stablizer}
$\fw_\chi=1$ for all $\chi\in \Omega_1$.
%\end{align}

%In the following we will mainly consider $\ggg=\gl, \frak{sl}, \osp$.

Recall  that the action of purely even subgroup $G_{\ev}$ on representations is given by
$\Psi^g(x)=\Psi(g^{-1}xg)$, where $\Psi: \ggg \rightarrow \gl(V)$ is a given representation of $\ggg$.
Let $T$ be a  maximal torus of $G_\ev$ with $\Lie(T)=\frak h$.  Then $\fw$ can be identified with $N_{G_\ev}(T)\slash C_{G_\ev}(T)$ (cf. \cite[\S24.1]{Hum}), and {$T$ stabilizes  $\Omega^{(1)}$ pointwise.} Thus, the action of $N_{G_\ev}(T)$ factors through the action of $\fw$. So we can consider the action of the Weyl group $\fw=N_{G_\ev}(T)\slash C_{G_\ev}(T)$ on representations of $\ggg$.

 \subsubsection{} Recall that associated with a regular semisimple $p$-character $\chi\in \hhh^*$, the baby Verma module $Z_\chi(\lambda)$ with $\lambda\in \Lambda(\chi)$ is an irreducible $U_\chi(\ggg)$-module (see \S\ref{Vermamodules}).
 In the sequel argument, {\sl{ we simply write $Z_\chi(\lambda)$ as $V(\lambda)$ for a given regular semisimple $p$-character $\chi$}}.
%Furthermore, we will denote by $V(\underline w(\lambda))$ the image of $V_\lambda$ under $w\in \fw$.
To emphasize the dependence of $V(\lambda)$ on a Borel subalgebra, we shall write
it by $V_{\frak{b}}(\lambda)$ for a moment.

 By the previous arguments, we can describe the action of $w\in \fw$ on baby Verma modules, by moving $V_{\frak{b}}(\lambda)$
into $V_{\frak{b}_{w}}(w(\lambda))$. Here $\frak{b}_w$ means the Borel subalgebra $w\cdot \frak{b}\cdot w^{-1}$ which equals to $\frak{h}+\sum_{\alpha\in \Delta^+}\ggg_{w(\alpha)}$. This is because if $v$ is a highest
vector in $V_{\frak{b}}(\lambda)$ for $\frak{b}$, then
$(h+n)v=\lambda(h)v, h\in\frak{h}, n\in\frak{n}^+$. So, in
$V_{\frak{b}_{w}}(w(\lambda))$, $(w(h+n)w^{-1})v=\lambda(h)v$.

Note that for a representation $\Psi$ of $\ggg$ with a regular semi-simple $p$-character, we have $\Psi(e_\alpha)^p=0$ for $\alpha\in \Delta_0$, and we have either $\Psi(e_\beta)^2=0$ for  $\beta\in \Delta_1$ with $2\beta\not\in \Delta$, or $\Psi(e_\beta)^2\in \bk\Psi(e_{2\beta})$ for $\beta\in\Delta_1$ with $2\beta\in \Delta_0$. So
$V_{\frak{b}_{w}}(w(\lambda))$ admits a unique $\frak{b}$-stable line.
Thus, there must exist some
$\lambda_w\in\frak{h}^{*}$ such that
$V_{\frak{b}_{w}}(w(\lambda))\cong V_{\frak{b}}(\lambda_w)$.

\iffalse
As what is forthcoming in \ref{sepc}, we consider those points as
representations $V(\lambda)$ with the highest weight $\lambda\in
\mathfrak{h}^{*}, \lambda^{p}-\lambda\in \Omega^{p}$.
%(Recall that
%the $V(\lambda)$ are induced from one-dimensional representations.)

In the case, $T$ stables all points of
$\Omega^{p}$, Whence the action of $N_{G_{\ev}}(T)$ factors to an
action of $\fw$. We write $V(\underline{w}\cdot\lambda)$, for the
image of $V(\lambda)$ under $w\in \fw$.
\fi

 We have  the
following  Lemma,  which will be important to Lemma \ref{win}.
\renewcommand{\labelenumi}{(\roman{enumi})}
\begin{lemma}\label{Wa}  Set $\fs(w)=\sum_{\alpha\in\Delta^+_{0}(w)}\alpha-\sum_{\beta\in\Delta^+_{1}(w)}\beta$ for $w\in \fw$, and
$\Delta^+(w)=\{\alpha\in\Delta^{+}\mid w^{-1}\alpha\in\Delta^{-}\}$, $\Delta^+_{i}(w)=\Delta^+(w)\cap\Delta_{i},i=0,1$. Then the following statements hold.
\begin{itemize}
\item[(1)]
$\lambda_w=w(\lambda)-\fs(w)$.
\item[(2)]
$\fs(w)=\rho-w(\rho)$.
%\item[(3)]
%If $w\in \textbf{Z}_{\fw}(\frak{h})$, then $\lambda_w=\lambda$.
\end{itemize}
\end{lemma}

\begin{proof}
The  proof is an analogue of that for  \cite[Lemma 4.3]{KW}. We give a complete proof here for the readers'  convenience because some new phenomena appear in the super case here,  such as the different meaning of $\rho$.

 Let $v$
be a nonzero eigenvector for $\frak{b}_{w}$ in
$V_{\frak{b}_{w}}(w(\lambda))$, on which the corresponding representation is denoted by $\Psi$ temporarily. Set
$$u=\prod\nolimits_{\beta\in\Delta^+_{1}(w)}
e_{\beta}\prod\nolimits_{\alpha\in\Delta^+_{0}(w)}e_{\alpha}^{p-1}.$$
Then
\begin{align*}
\Psi^w(u)v&=(\prod\nolimits_{\beta\in\Delta^+_{1}(w)}w^{-1}
e_{\beta}w\prod\nolimits_{\alpha\in\Delta^+_{0}(w)}w^{-1}
e_{\alpha}^{p-1}w)v\cr
&=C(\prod\nolimits_{\beta\in\Delta^+_{1}(w)}e_{w^{-1}(\beta)}
\prod\nolimits_{\alpha\in\Delta^+_{0}(w)}e_{w^{-1}(\alpha)}^{p-1})v,
 \end{align*}
 which is denoted by $v'$ with $C$ a nonzero number in $\bk$. By definition, we know that $v'$ is actually a nonzero vector of $V_{\bbb_w}(w(\lambda))$.
It is not difficult to check that $v'$
is a nonzero $\frak{b}$-eigenvector with highest weight $\lambda_w=w(\lambda)-\fs(w)$.
Hence, Statement (1) is proved.

 By induction beginning with $\fs(s_\alpha)=\rho-s_\alpha(\rho)$ we can easily prove (2).
% The third statement follows directly from the definition of $\lambda_w$.
  %Combining (1) and (2), we get $\lambda_w=w(\lambda)+w(\rho)-\rho$. The condition $w\in \textbf{Z}_\fw(\hhh)$ implies that $\lambda_w=\lambda$.
\end{proof}

\subsubsection{}\label{sec: Harish-Chandra} Next we consider the Harish-Chandra homomorphism $\gamma:U(\ggg)\rightarrow U(\hhh)$ as in \S\ref{sec: harish-chandra}.
%, which is by definition the composite of the canonical projection %$\gamma_1: U(\ggg)\rightarrow U(\hhh)$ and the algebra homomorphism %$\beta:U(\hhh)\rightarrow U(\hhh)$ defined via $h\mapsto h+\rho(h)$ for %all $h\in \hhh$.
In the following arguments, we will identify $U(\hhh)$ (=the symmetric algebra $\mf S(\hhh)$) with  $\bk[\hhh^*]$ the polynomial ring over $\hhh^*$. The following lemma is an analogue of \cite[Lemma 5.1]{KW}, the proof of which can be done by the same arguments as in \cite{KW}, omitted here.

\begin{lemma} (cf. \cite[Lemma 5.1]{KW})\label{mon} The restriction map of $\gamma$:
 $U(\ggg)^{T}\rightarrow U(\mathfrak{h})$ is a
homomorphism of algebras, still denoted by $\gamma$.
\end{lemma}

Furthermore, we have the following

\begin{lemma}\label{win}
For $w\in \fw$, $n_{w}\in N_{G_{\ev}}(T)$ a representative of $w\in \fw$,
we have $\gamma(n_{w}zn_{w}^{-1})=w\gamma(z)w^{-1}$, for all $z\in
\cz$.
In particular, $\gamma(\cz^{N_{G_{\ev}}(T)})\subseteq U(\frak h)^{\fw}$.
\end{lemma}

\begin{proof} Maintain the notations as previously (in particular, as in the arguments prior to Lemma \ref{Wa}). Recall that for $u\in U(\ggg)$, the PBW theorem enables  us to write uniquely
$u=\varphi_{0}+\sum u_{i}^{-}u_{i}^{+}\varphi_{i}$, such that
$\varphi_{i}\in U(\hhh),u_{i}^{\pm} \in
U(\mathfrak{n}^{\pm})$, and  $u_{i}^{-}u_{i}^{+}\neq 0$. And then $\gamma_{1}(u)= \varphi_{0}$. And $\gamma=\beta\circ\gamma_{1}$.
%Let $z=\varphi_{0}+\sum u_{i}^{-}u_{i}^{+}\varphi_{i}\in\mathcal
%{Z},u_{i}^{-}u_{i}^{+}\neq 0$.
Recall that for  $z\in\cz$, $z\in U(\ggg)_\bz$ (cf. Lemma \ref{cev}). Hence $z$ acts as a scalar $\chi_\lambda(z)$ on the irreducible module $V_{ \frak b}(\lambda)$. In particular,
$zv_\lambda=\chi_{\lambda}(z)v_\lambda$ for a highest weight vector $v_\lambda$ in $V_\bbb(\lambda)$. On the other hand,
$zv_\lambda=\varphi_{0}(\lambda)v$. Here we identify $U(\hhh)$ with $\mf S(\hhh)$. So
$\chi_{\lambda}(z)=\varphi_{0}(\lambda)=\gamma_{1}(z)(\lambda)$. According to the argument prior to Lemma \ref{Wa}, the image of $V_{\frak b}(\lambda)$ under the action
of $w$ is isomorphic to $V_{ \frak b}(\lambda_w)$, and
$n_{w}zn_{w}^{-1}$ acts on $V_{ \frak b}(\lambda_w)$ as a scalar
$\chi_{\lambda_w}(z)=\varphi_{0}(\lambda_w)$.
 Then,
$\gamma(n_{w}zn_{w}^{-1})(\lambda)=\varphi_{0}(\lambda_w+\rho)=\varphi_{0}(w(\lambda+\rho))
= w\gamma(z)w^{-1}(\lambda)$. Hence $\gamma(n_{w}zn_{w}^{-1})=w\gamma(z)w^{-1}$.
\end{proof}

\subsubsection{} We will still denote by $\Frac(R)$ the fraction field of an integral domain $R$.
 We will consider the subalgebra
$U(\ggg)^{G_{\ev}}\cap \cz$ which is just $\cz^{G_{\ev}}$, a large part of the whole center.   By Proposition \ref{cev},
$\cz$ is integral over $\cz_p$. The following result is clear.

\begin{lemma}\label{prip}
For each $z\in \cz^{G_{\ev}}$, there exists a unique monic
polynomial { $g_0(x)\in \cz_p[x]$ }with  minimal degree among those polynomials taking $z$ as a root.
\end{lemma}

\iffalse
\begin{proof}
Let $z$ be a nonzero element of $\cz^{G_{\ev}},z\neq 0$. If $z\in  \cz_p$, the statement holds obviously.
 Assume $z\not\in \cz_p$. By Proposition \ref{cev},
$\cz$ is integral over $\cz_p$. There exists a monic polynomial
$f(x)=x^{d}+a_{1}x^{d-1}+\cdots+a_{d-1}x+a_{d},a_{i}\in
\Frac(\cz_p)$ which is primitive such that $f(z)=0$. If
there exists another polynomial $g(t)$ with the same degree such
that $g(z)=b_{0}z^{d}+b_{1}z^{d-1}+\cdots+b_{d-1}z+b_{d}=0$,
$b_{i}\in \cz_p$. then
$(b_{0}f-g)(z)=(b_{0}a_{1}-b_{1})z^{d-1}+\cdots+(b_{0}a_{i}-b_{i}))z^{d-i}+\cdots+(b_{0}a_{d-1}-b_{d-1}))z+(b_{0}a_{d}-b_{d}))=0$.
$\cz_p$ is a unique fractional integral ring, there
exists a primitive polynomial with minimal degree
$h(t)=c_{0}z^{m}+c_{1}z^{m-1}+\cdots+c_{m-1}z+c_{m}=0, c_{i}\in
\cz_p,  c_{0}\neq 0$, and such polynomial is unique up to
the element of $\Frac(\cz_p)$. Therefore we can find the
polynomial $p(t)=\frac{1}{c_{0}}h(t)$ satisfying the condition.
\end{proof}
\fi

We can obtain the following proposition.
\begin{prop}\label{inj}
$\gamma: \cz^{G_{\ev}}\rightarrow U(\frak h)^{\fw}$ is an injective
homomorphism of algebra.
\end{prop}
\begin{proof}

We know that $\gamma$ is an algebra homomorphism by Lemma
\ref{mon}, and $\gamma(\cz^{G_{\ev}})\subseteq U(\frak
h)^{\fw}$  by Lemma \ref{win}. Therefore it is enough to verify that
$\gamma$ is injective.
Take a nonzero element $z\in\cz^{G_{\ev}}$. We will prove that $\gamma(z)\ne 0$.
According to Lemma \ref{prip}, there
exists a unique monic polynomial $g_0(x)$ with minimal  degree such
that $g_0(z)=z^{m}+a_{1}z^{m-1}+\cdots+a_{m-1}z+a_{m}=0, a_{i}\in
\Frac(\cz_p)$.
By \cite[Lemma 4.6]{KW}, we know $\Frac(\cz_p)^{G_{\ev}}=\Frac({\cz_p}^{G_{\ev}})$. So we immediately get that all those $a_i\in \Frac({\cz_p}^{G_\ev})$.
 As to the restriction on ${\cz_p}^{G_{\ev}}$ of $\gamma$, it is known by the arguments in \cite{KW} that  $\gamma:{\cz_p}^{G_\ev}\rightarrow U(\frak h)^{\fw}$ is injective. By a natural extension, it is easy to check that $\gamma$ is also injective on
$\Frac(\cz_p^{G_{\ev}})$.
\iffalse
By \cite{KW}, we get $\gamma|_{\cz_p}=\gamma_{1}$. Moreover restricting on  Now we extend $\cz_p^{G_{\ev}}$ to
$\Frac({\cz_p^{G_{\ev}}})$ such that $\frac a
b\longmapsto \frac
{\gamma(a)}{\gamma(b)}, a,b\in\cz_p^{G_{\ev}}$. %extending
%to $\Frac({\cz_p^{G_{\ev}}})$.
\fi
Note that
$\sigma(a)=a$ for $\sigma\in G_{\ev}$ and for $a\in \Frac(\cz_p)^{G_{\ev}}=\Frac({\cz_p}^{G_{\ev}})$. Applying the algebra homomorphism $\gamma$ to that equation, we have
$\sum\limits_{i=1}^{m}\gamma(a_{m-i})\gamma(z)^{i}+\gamma(a_{m})=0$.
And then, $\gamma(z)(\sum\limits_{i=1}^{m}\gamma(a_{m-i})\gamma(z)^{i-1})=-\gamma(a_{m})$. %\neq0
By the arguments in \S\ref{CentralPrim},  $\cz^{G_{\ev}}$ has no zero-divisor. We claim that the constant
term $a_{m}\neq 0$. Otherwise,
$g_1(x):=x^{m-1}+a_{1}x^{m-2}+\cdots+a_{m-1}$ is a nonzero polynomial with $g_1(z)=0$, a contradiction with the assumption of minimal degree. This implies  that $\gamma(a_m)\ne 0$. Hence, $\gamma(z)\neq0$.
\end{proof}

\iffalse
\begin{remark}
By \cite[Theorem A.1.1]{MFK}, $\cz$ is finitely generated
$\cz_p$-module, we can get $\cz^{G_{\ev}}$  is
finitely generated ${\cz_p}^{G_\ev}$-module. In accordance
with Proposition \ref{cev} again, $\cz^{G_{\ev}}$ is integral
over ${\cz_p}^{G_\ev}$. Proposition \ref{inj} is also
obtained.
\end{remark}
\fi
%From now on we will always assume
%\[
%(*)\left\{
%\begin{array}{l}
%\cz~\mbox{has no zero-divisor in} ~U(\ggg)~, \mbox{Denote by
%} ~\bbf=\Frac(\cz)~\mbox{the fraction ring of}
%~\cz \\
%\mbox{Define the fraction ring of} ~U(\ggg)~ \mbox{over}
%~\bbf, ~\mathcal
%{D}(\ggg)=U(\ggg)\otimes_{\cz}\Frac ({\cz})~,\mbox{and it is a}\\
%\mbox{simple superalgebra}.
%\end{array}
% \right.
%\]
%\begin{remark}
%
%We conjecture that the above assumption $(*)$ holds for basic
%classical Lie superalgebras. Moreover, we will see that the
%assumption holds for $B(0|n)$ type Lie superalgebra( see
%\ref{sdom}).
%\end{remark}

\subsection{The structure of center: proof of Theorem \ref{secondtheorem}}
%Maintain the notations as above.
Maintain the notations as before. In particular,  $\bbf$ denotes the fraction field  of $\cz$. And $\cd(\ggg)=U(\ggg)\otimes_{\cz}\bbf$, the fraction ring of $U(\ggg)$ over $\bbf$. The remaining part from  \S\ref{sec: begin proof} through to \S\ref{sec: end proof} is devoted to the proof of Theorem \ref{secondtheorem}.

\subsubsection{}\label{sec: begin proof} By Proposition \ref{cev}, $\cz$ is a finitely-generated commutative algebra over $\bk$.
So we have the Zassenhaus variety $\mbox{Spec}(\cz)$ for $\ggg$, analogous to the ordinary Lie algebra case,
which is defined to be the maximal spectrum of $\cz$. We can identify $\mbox{Spec}(\cz)$
with {\sl the affine algebraic variety of algebra morphisms $\cz$ to $\bk$, the latter of which will be denoted by $\Spec(\cz)$}.
%Following Kac-Weisfeiler's arguments in \cite{KW}, we first look at the action of Weyl group $\fw$ on points of
%$\mbox{Spec}(\cz)$ over $\Omega^{p}\subseteq \mbox{Spec}(\cz_p)$.

{
Denote by {\sl $\Spec (U(\ggg))$  the set of iso-classes of irreducible $U(\ggg)$-modules}.}
Naturally, there is a $G_{\ev}$-equivariant map
$$\pi:\Spec (U(\ggg))\rightarrow \Spec(\cz),$$
 which is defined by the central
characters over irreducible modules.

\subsubsection{}\label{ccc} Let us  first give the following  observation, due to J. Wei.
 \begin{lemma}\label{lem: no nonzero divisors} (J. Wei)
  Let $\mathcal{D}(\ggg)$ be the fraction ring of $U(\ggg)$ as in (\ref{eq: fractional alg}). Then $\mathcal{D}(\ggg)$ is a finite-dimensional central simple algebra over $\bbf$.
 \end{lemma}

 \begin{proof} This is because $U(\ggg)$ is a finitely generated $\cz$-module (note that $U(\ggg)$ is a free $\cz_p$-module of finite rank), and therefore a PI ring (cf. \cite[Corollary 13.1.13]{McRo})
 %\footnote{We thank J. Wei who made us aware of this reference.}.
 Hence from \cite[Theorem 15.4.2]{Musson} and Posner's Theorem (cf. \cite[Theorem 13.6.5]{McRo}).  The lemma  follows.
 \end{proof}

Set $q=\dim_{\bbf} \mathcal {D}(\ggg)$.
Consider an
 irreducible $\mathcal {D}(\ggg)$-module $V_\bbf$ over $\bbf$. For a given nonzero vector $v\in V_\bbf$, we can write
 $V_\bbf=\mathcal{D}(\ggg)v$.
 Assume $\dim V_{\bbf}=l$, then $q=l^2$. Then we can assume that there is  a set of $\bbf$-basis
$\{v_{1}=v, v_{2}, \cdots, v_{l}\}$ of $V_\bbf$, which are included in $U(\ggg)v$.
%By Proposition
%\ref{cev},
Recall that
 $U(\ggg)$ is a finitely-generated $\cz$-module. We assume the $\{d_{1}, \cdots, d_{m}\}$ is  a set of generators of $U(\ggg)$ over $\cz$.
For each pair $(d_{i}, v_{j})$, there exists $z_{ij}\in\cz$, such
that $z_{ij}d_{i}v_{j}\in \sum_t\cz v_{t}$. Denote by
$\mathcal {C}$ the product of all the central elements $z_{ij}$.
Then $V_{\mathcal C}:=U(\ggg){\mathcal C}v\subset \sum_{i}
U(\ggg){\mathcal C} v_{i}\subset \sum_{t} \cz v_{t}$.

For each $u\in U(\ggg)$, we have $u(\mathcal
{C}v_{1},\cdots,\mathcal {C}v_{l})=(v_{1},\cdots, v_{l})U$, $U$ is
the matrix with each entry in $\cz$. So we can define
$\Psi(u)=\mbox{Tr}(U)\in\cz$ and the bilinear form
$B(u_{1},u_{2})=\Psi(u_{1}u_{2})$.

Given $x_1,\cdots,x_q\in U(\ggg)$, consider $q$th determinant
$\frak D:=\Det(B(x_i,x_j)_{q\times q})$, the ideal which generates
in $\cz$ is called discriminant ideal. Recall that $\mathcal  D(\ggg)$ is a central simple algebra over $\bbf$ of dimension $q$. There must exist a nonzero discriminant ideal in $\cz$.
Naturally, $B(\cdot,\cdot)$ can be extended to a non-degenerate bilinear form in $\mathcal{D}(\ggg)$.

\begin{defn}\label{det}
Define $\frak A\subset \Spec(\cz)$ to be the set
$\{ \varphi\in \Spec(\cz) \mid$  there exist $x_1,\cdots,x_q\in U(\ggg)$, such that $\frak D=\Det(B(x_i,x_j)_{q\times q})$ satisfies
$\varphi(\mathcal {C} \frak D)\ne 0\}.$
\end{defn}

By the arguments above, $\frak A$ is a non-empty open subset of
$\Spec(\cz)$.
\begin{lemma}\label{geb}
%Assume the assumption $(*)$ holds.
There exists an open dense
subset $W$ of $\Spec(\cz)$, such that the map $\pi_{W}:
\pi^{-1}W\rightarrow W$ induced from $\pi$ is bijective.
\end{lemma}

\begin{proof} We will divide several steps to find a desirable $W$. Maintain the notations as in the arguments before Definition \ref{det}.

We first consider the fiber of $\pi$ over $\varphi$ for $\varphi\in \frak A$.
According to Definition of $\frak A$, there exist $x_1,\cdots,x_q\in
U(\ggg)$ such that $\varphi(\mathcal {C}\frak D)\ne 0$, where
${\frak D}=\Det(B(x_i,x_j)_{q\times q})$. Set
$U_{\varphi}:=U(\ggg)\otimes_{(\cz,\varphi)}\bk$, and
$V_{\bk}=V_{\mathcal {C}}\otimes_{(\cz,\varphi)} \bk$. Here the tensor products are defined over $\cz$, and  $\bk$ is regarded
as a $\cz$-module induced from $\varphi$. Set $\bar
u$ for $u\in U(\ggg)$ to be the image of $u$ under the map
$-\otimes_{(\cz,\varphi)}1: U(\ggg)\rightarrow U_\varphi$.

We then claim that $\bar {x}_i, i=1,\cdots, q$, form a basis of
$U_{\varphi}$.

In fact, by the choice of $x_1,\cdots, x_q$, $\varphi(\frak
D)\ne 0$. So $x_i,i=1,\cdots, q$ are $\bbf$-linear independent
elements in $\mathcal {D}(\ggg)$. Then $x_i,i=1,\cdots, q$ is a
$\bbf$-basis of $\mathcal {D}(\ggg)$. As $B(\cdot,\cdot)$ is non-degenerate, we can take a dual basis $\{y_i,i=1,\cdots,
q\}$ of the basis $\{x_i,i=1,\cdots, q\}$ in $\mathcal
{D}(\ggg)$ with respect to this bilinear form. Thus, ${\frak
D}y_i\in U(\ggg)$. Furthermore, for each $u\in U(\ggg)$,
$${\frak D}u=\sum_{i=1}^qB(\frak {D}u, y_i)x_i.$$
Since $\varphi(\frak D)\ne 0$, and $B(\frak {D}u, y_i)=B(u,\frak
{D}y_i)\in \cz$, we have $\bar u=\sum_{i=1}^q \varphi(\frak
D)^{-1}$ $\varphi(B({\frak D}u,y_i))\bar x_i$. That is to say, $U_\varphi$ is $\bk$-spanned by $\{\bar
x_i,i=1,\cdots,q\}$. Furthermore, $\Det(B_\varphi(\bar
x_i,\bar x_j)_{q\times q})=\varphi(\frak D)\ne 0$ where
$B_{\varphi}$ is $\bk$-valued of $B$ through $\varphi$, $B_{\varphi}$
then naturally becomes a non-degenerate trace form of $U_{\varphi}$
defined by the module $V_\bk=V_{\mathcal {C}}\otimes_{(\cz,\varphi)}\bk$. The claim is proved. Moreover, $\dim_\bk V_\bk=l$.

Next, we claim that $U_\varphi$ is a semisimple algebra, naturally a semisimple superalgebra (cf.
\cite{BK} or \cite{Jo}). Note that $\mathcal{D}(\ggg)$ is a central simple algebra over $\bbf$, as observed first in \S\ref{ccc}. By the same arguments as used in  the proof of \cite[Theorem 5]{Zass}, we can get if $\bar x$ is
a nonzero element in the Jacobson radical of $U_\varphi$, then for
each $\bar u\in U_\varphi$, $B_\varphi(\bar x, \bar u)=0$. The non-degeneracy of $B_\varphi$ ensures that the  Jacobson radical of $U_\varphi$ is trivial. So the claim on the semisimple property of $U_\varphi$ is proved.

Furthermore, we note that if the semisimple superalgebra $U_\varphi$ is simple
as an associative algebra, it has the form $\mbox{Mat}_m(\bk)$. So $\rho:
U(\ggg)\rightarrow \mbox{Mat}_m(\bk)$ is a representation of $U(\ggg)$. And
all irreducible representations of $\ggg$ induced  from $\varphi$ in this way
are isomorphic to $V_\bk$, naturally having the same central character.
% On the surjective part is obvious, we skip it.

\iffalse
As $\mathcal{D}(\ggg)$ is a central simple algebra over $\bbf$,   then $q=\dim_\bk U_\varphi=\dim_\bbf \mathcal {D}(\ggg)$ is a square number. Assume $q=l^2$. Then $\dim V_\bk=l$.
\fi

Recall that $\dim_\bk U_\varphi=q=l^2$ by the arguments above, and that $\dim_\bk V_\bk=l$ . To ensure that the semisimple superalgebra $U_\varphi$ is simple, it is enough
to ensure that  $V_\bk$ is a simple $U_{\varphi}$-module. For this, we  take an irreducible
submodule  $V_{1}$ of $V_\bk$. Note that $V_{\bk}=U_\varphi v$. We can then take a nonzero element in $V_1$, written as $fv$ with $f\in U_{\varphi}$. Then $V_1=U_\varphi  fv$. There exists $R\in U(\ggg)$,
satisfying $\overline{R}=f$. Observe that for
$\widetilde{V}_{1}:=U(\ggg)R\mathcal {C}v \subset V_{\mathcal {C}}$, $\widetilde{V}:=\widetilde {V_1}\otimes_{\cz} \bbf$ is
a non-trivial submodule of  $\mathcal {D}(\ggg)$ in $V_{\bbf}$. Hence, $\widetilde
V_1$ and $V_{\mathcal{C}}$ are two $\cz$-lattices of
$V_{\bbf}$. Since $V_{\bbf}$ is an irreducible $\mathcal{D}(\ggg)$-module over $\bbf$ and $\widetilde{V}_1\subset V_{\mathcal{C}}$, there must exist
$P\in \cz$ such that $\widetilde{V_{1}}=P
V_{\mathcal{C}}$. Hence
$V_1=\widetilde{V_{1}}\otimes_{(\cz,\varphi)}\bk=PV_{\mathcal{C}}\otimes_{(\cz,\varphi)}\bk$ (note that $\varphi(\mathcal{C})$ is a nonzero number in $\bk$ and $\overline R=f\ne 0$). We get  $\varphi(P)\ne 0$, hence $V_1=V_\bk$. So, $V_\bk$ is a simple $U(\ggg)$-module arising from $\varphi$, which is a unique simple module of the simple superalgebra $U_\varphi$, up to isomorphism, because all simple (super) modules of the central simple algebra $\mathcal{D}(\ggg)$ are isomorphic.% (cf. Remark \ref{simplesuper}).
 Therefore, it suffices to take $W$ equal to $\frak A$. Then it is a desirable open set of $\Spec(\cz)$.
\end{proof}

\begin{remark}\label{regt}
That $\mathcal {C}$ appearing in the first paragraph of \S\ref{ccc}  is related to  the
choice of irreducible $\mathcal {D}(\ggg)$-module $V_{\bbf}$.
We can choose appropriate $V_{\bbf}$ such that $\mathcal
{C}=1$. In fact, $U(\ggg)$ itself is a $\cz$-lattice in the regular $\mathcal {D}(\ggg)$-module. Hence the central simple (super-) algebra $\mathcal{D}(\ggg)$ over $\bbf$ can be decomposed into the direct sum of some (left) regular irreducible supermodules as:  $\mathcal {D}(\ggg)=\bigoplus_{i}U_{i}$ (cf. \cite[Proposition 2.11]{Jo}). Then $U(\ggg)\cap U_{i}$ becomes a
$\cz$-lattice. In such a case, $\mathcal {C}=1$.
\end{remark}

\subsubsection{} We need to understand the field extension $\Frac(\cz):\Frac(\cz_p)$.  The following theorem is very important to the latter discussion.

\begin{theorem} (cf. \cite[Theorem 5.1.6]{Sp}\label{sep})
Let $X$ and $Y$ be irreducible varieties and let $\varphi:
X\rightarrow Y$ be a dominant morphism. Put $r=\dim X-\dim Y$,
there exists non-empty open subset $U$ of $X$ with the following
properties:
\begin{itemize}
\item[(1)]
Restricting on $U$, $\varphi$ is open morphism $U\rightarrow Y$.
\item[(2)] {
If $Y^{'}$ is an irreducible closed subvariety of $Y$, $X^{'}$ is
an irreducible component of $\varphi^{-1}(Y^{'})$ intersecting with
$U$.} Then $\dim X^{'}=\dim Y^{'}+r$. In particular,  if $y\in Y$,
any irreducible component of $\varphi^{-1}(y)$ that intersects $U$
has dimensional $r$.
\item[(3)]
If $\bk(X)$ is algebraic over $\bk(Y)$, then for each $x\in U$, the
dimension of the points in the fiber of $\varphi^{-1}(\varphi(x))$
is equal to the separable degree $[\bk(X),\bk(Y)]_{s}$ of $\bk(X)$ over
$\bk(Y)$.
\end{itemize}
\end{theorem}

\begin{lemma}\label{sep1} Maintain the notations as previously. The separable degree of
$\Frac(\cz)$ over $\Frac(\cz_p)$ is not smaller
than $p^{r}$, $r=\dim\mathfrak{h}$.
\end{lemma}
\begin{proof}
\iffalse
In view of the paper \cite{CKP}, our method is analogous to one part
of \cite[Lemma 4.2]{KW} in the proof of ``inequality of extension
degree"(the other part of that proof is not applicable in our case)

Let $\Spec (U(\ggg))$ denote the set of all $G_{\ev}$-equivalent
classes of finite dimensional irreducible representations of
$U(\ggg)$.
\fi
By Proposition \ref{cev}, $\cz$ is integral over $\cz_p$. Thus we have a surjective map of irreducible varieties $\phi: \Spec(\cz)\rightarrow \Spec(\cz_p)$, induced from the  inclusion
$\cz_p\hookrightarrow \cz$  (cf. \cite[ChV. \S1. Ex.3]{Sha} for surjective property).
 { Recall that under $\textsf{Ad}$-action, $G$ becomes a subgroup of the restricted automorphism group of $\ggg$, i.e. $\textsf{Ad}(g)$ for $g\in G$ becomes an automorphism of $\ggg$ commutating with the $p$-mapping $[p]$ of $\ggg$ (see \cite[\S I.3.5, \S I.3.19]{Bor} or \cite{SZheng}). So $\textsf{Ad}(G)$  naturally acts on $\Spec(U(\ggg))$, $\Spec(\cz)$ and $\Spec(\cz_p)$, respectively.
 Consider the canonical $G_{\ev}$-equivariant mapping
sequence:
$$\Spec (U(\ggg))\overset{\pi}{\longrightarrow}
\Spec(\cz)\overset{\phi}{\longrightarrow}\Spec(\cz_p),$$
where  $\pi:\Spec(U(\ggg))\rightarrow \Spec(\cz)$ is defined from the central
character associating to an irreducible representation.} By lemma
\ref{geb}, $\pi$ is generically bijective, that is, there exists an
open dense subset  $W$ of $\Spec (\cz)$ such that $\pi:
\pi^{-1}(W)\rightarrow W$ is bijective. Put
$U=\phi^{-1}(G_{\ev}\cdot \Omega)$. Note that $G_\ev\cdot \Omega$ is an open set of ${\ggg_\bz}^*$,  thereby regarded as an open set of $\Spec(\cz_p)$. Hence $U$ becomes an open set of $\Spec(\cz)$. Consider $S=U\cap W$. It is an
open dense subset of $\Spec(\cz)$. For $\chi\in
G_{\ev}\cdot\Omega$, we might as well take $\chi\in \Omega$ without loss of generality. Recall that the set $\{Z_\chi(\lambda)\mid \lambda\in \Lambda(\chi)\}$ is just a complete set of $G_\ev$-equivalent classes of irreducible modules for $U_\chi(\ggg)$ (cf. \cite[Corollary 5.7]{WZ}). The number of the points in $\phi^{-1}(\chi)$ is bigger or
equal to $ p^r$. Since $\phi$ is finite dominant, by Theorem
\ref{sep}(3), the separable degree of $\Frac(\cz)$ over
$\Frac(\cz_p)$ is not smaller than $p^r$.
\end{proof}

Assume $\dim\ggg_{\bar 0}=s$,  $\dim\ggg_{\bar 1}=t$. Then $t$ is even in
our case.

For the convenience of arguments in the sequel, we will continue to introduce the central-valued function $\Theta$ associated to a given pair $(\chi,\lambda)$ for $\mathcal{D}(\ggg)$ and for $U(\ggg)$, as we have done in the arguments used in Lemma  \ref{geb}, where $\chi\in \Omega$ and $\lambda\in  \Lambda(\chi)$.  %This $\Theta$
%is somewhat a concretization  of the specialization introduced by
%Zassenhaus in \cite[Pages 25-26]{Zass} which lays a foundation for the %modular
%representation theory of Lie algebras.

Keep the assumption as above for $(\chi,\lambda)$. Let us introduce $\Theta$, which is by definition  the canonical algebra homomorphism from $U(\ggg)$ to
$U(\ggg)\slash(z-\Theta_{(\chi,\lambda)}(z))$, where
$(z-\Theta_{(\chi,\lambda)}(z))$ means the ideal of $U(\ggg)$
generated by $z$ taking all through $\cz$, and
$\Theta_{(\chi,\lambda)}(z)$ is defined by $z\cdot
v_{\lambda}=\Theta_{(\chi,\lambda)}(z)v_{\lambda}$ with $v_{\lambda}$ being the canonical generator of the baby Verma module  $Z_\chi(\lambda)$. {Keep it in mind that $\Theta_{(\chi,\lambda)}(z)$ is actually  the same as the usual Harish-Chandra map (see the proof of Lemma \ref{win}).}  The category of
$\Theta(U(\ggg))$-modules is a subcategory of
$U_{\chi}(\ggg)$-modules.

{
\begin{lemma}\label{lem: central simple} $\Theta(U(\ggg))$ is a central simple superalgebra over $\theta(\cz)$ ($=\bk$).
\end{lemma}

\begin{proof} By definition, $\Theta(U(\ggg))=U(\ggg)\slash(z-\Theta_{(\chi,\lambda)}(z))\cong U_\chi(\ggg)\slash (\bar z-\Theta_{(\chi,\lambda)}(z))$  with $\chi\in \Omega$ and $\lambda\in \Lambda(\chi)$,  where $\bar z$ is  the image of $z$ in $U_\chi(\ggg)$ for $z$ running through $\cz$. Note that $\chi$ is a regular semisimple $p$-character. By \cite[Corollary 5.7]{WZ}, $U_\chi(\ggg)$ is a semisimple superalgebra, which has $p^r$ different irreducible modules $Z_\chi(\lambda)$, $\lambda \in \Lambda(\chi)=\{\lambda_i\mid i=1,\ldots,p^r\}$. So $U_\chi(\ggg)$ is decomposed into $p^r$ blocks $U_\chi(\ggg)=I_{\lambda_1}\oplus \cdots\oplus I_{\lambda_{p^r}}$. Each $I_{\lambda_i}$ has a unique irreducible module which is isomorphic to $Z_\chi(\lambda_i)$ with dimension $q:=p^{\dim \nnn_\bz^-}2^{\dim \nnn_\bo^-}$. Let $e$ be an indecomposable central  idempotent of $U_\chi(\ggg)$. Then $eU_\chi(\ggg)$ is just a block.  Each block is identical to $eU_\chi(\ggg)$ for some indecomposable central idempotent $e$, which has a central subalgebra $e\overline\cz$ where $\overline\cz$ denotes the image of $\cz$ in $U_\chi(\ggg)$. Such a block $I_\lambda$  for some $\lambda\in \Lambda(\chi)$  is of type $\textsf{M}$ (see \cite{Jo},  the so-called type $\textsf{M}$ just means the usual matrix superalgebra $\textsf{M}_{c|d}$ over $\bk$ with $\dim\textsf{M}_{c|d}=(p^{\dim \nnn_\bz^-}2^{\dim \nnn_\bo^-})^2$).  On the other side, by the semi-simplicity of $U_\chi(\ggg)$ mentioned above, $U_\chi(\ggg)\cong \bigoplus p^r\textsf{M}_{c|d}$ where the coefficient $p^r$ stands for the  direct sum of $p^r$  copies of $\textsf{M}_{c|d}$. Note that $\dim U_\chi(\ggg)=p^{\dim\hhh+2\dim \nnn^-_\bz}2^{2\dim\nnn_\bo^-}$  and then
$$\dim U_\chi(\ggg)\slash \dim\textsf{M}_{c|d}=p^{\dim\ggg_\bz}2^{\dim\ggg_\bo}\slash (p^{\dim \nnn_\bz^-}2^{\dim \nnn_\bo^-})^2=p^r.$$
 Let $C_e$ be the centralizer of $eU_\chi(\ggg)$ in $U_\chi(\ggg)$. Then $C_e$ contains $e\overline\cz$. Hence $C:=\sum_{e}C_e$ with $e$ running indecomposable central idempotents, contains $\overline\cz$. The dimension of $C$ is  $p^r$, which coincides with $\dim \overline\cz$, thanks to Lemma \ref{sep1}.  Hence $U_\chi(\ggg)\cong \textsf{M}_{c|d}(\overline\cz)$. Consequently, $\Theta(U(\ggg))\cong \textsf{M}_{c|d}(\Theta(\cz))$ by the analysis in the beginning of the proof. As to the coincidence between $\Theta(\cz)$ and $\bk$, this is by definition of $\Theta_{(\chi,\lambda)}$ (recall that $\Theta_{(\chi,\lambda)}(z)$ is given by Harish-Chandra map as usually). This completes the proof.
\end{proof}
}

\begin{prop}\label{sepc}  $\Frac(\cz)$
is separable over $\Frac(\cz_p)$, and $[\Frac(\cz): \Frac(\cz_p)]=p^{r}, r=\dim\mathfrak{h}$.
\end{prop}
\begin{proof}
On one side, we fix $\chi\in \Omega$, a regular semisimple $p$-character. Then, associated with $\chi$, the baby Verma module $Z_\chi(\lambda)$  is an irreducible $U_{\chi}(\ggg)$-module (cf. \cite[Corollary 5.7]{WZ}),
and
\begin{eqnarray}\label{Babydim}\dim Z_\chi(\lambda)=p^{(\dim\ggg_\bz-r)/2}2^{\dim\ggg_\bo/2}=p^{(s-r)/2}2^{t/2}.
\end{eqnarray}
\iffalse
Since $\mathcal {D}(\ggg)$ is a finite-dimensional $\bbf$-superalgebra, it has finite iso-classes of irreducible modules. Assume  the maximal dimension of those irreducible modules to be
$l$, and assume that $V_\bbf$ is an  irreducible $\mathcal
{D}(\ggg)$-module of dimension $l$.
\fi
On the other side, we recall that $\mathcal {D}(\ggg)$ is a central simple
$\bbf$-algebra (Lemma \ref{lem: no nonzero divisors}), all simple $\mathcal
{D}(\ggg)$-modules are isomorphic. Still assume that $V_\bbf$ is a simple $\mc{D}(\ggg)$-module with dimension $l$, which  admits a $\cz$-lattice
$V=U(\ggg)v$ (cf. Remark \ref{regt}). We may write $V_\bbf=\mc{D}(\ggg)v$.
Thus, given $\bk$-valued $\Theta$ for
$\cz$ as previously,  we can take $V_{\bk}=\Theta(U(\ggg))v$.
In this way, we actually have had a module-operator $\Theta^{V}: V\rightarrow
V_{\bk}$, satisfying $\Theta^{V}(uw)=\Theta(u)\Theta^{V}(w)$,
$\Theta^{V}(v)=v$, for each $u\in U(\ggg)$, $w\in V$.

\renewcommand{\labelenumi}{(\arabic{enumi})}
%Note that $\mathcal {D}(\ggg)$ is simple $\bbf$-superalgebra.
Recall that as in Remark \ref{regt}, $\mathcal {D}(\ggg)$ is the sum of
isomorphic simple left modules, $\mathcal
{D}(\ggg)=\sum_{i=1}^{l}V_{i}$, $V_{i}=\mathcal {D}(\ggg)u_{i}$, and
$V_{\bbf}$ is isomorphic to each $V_i$. In the following arguments, we might as well assume $V_{\bbf}=V_{1}$. Thus,
$\Theta(u_1)\neq 0$, and $V_{\bk}=\Theta(U(\ggg)u_1)$. Next we claim the dimension relations as below:

\begin{eqnarray}\label{dim}
\dim_\bbf\mathcal {D}(\ggg)\geq\dim_{\bk}\Theta(U(\ggg))\geq\dim_\bk
\rm {End}_{\bk}(V_\bk)\geq(\dim_{\bk} Z_\chi(\lambda))^2.
\end{eqnarray}

Let us prove it below.  As to the first inequality, we recall again that $\mathcal{D}(\ggg)$ is a central simple (super)algebra over $\bbf$ with center being even. By Wall's theorem (see \cite{Wa} or \cite[Theorem 2.6]{Jo}), $\mathcal{D}(\ggg)$ is isomorphic to $\textsf{M}_{c'|d'}$ for some $c',d'\in\bbn$ with dimension $q=l^2$, which is a usual matrix superalgebra over $\bbf$ with usual matrix product. We can define  the minimal polynomial degree $\wp_\bbf(\mathcal{D}(\ggg))$ of $\mathcal{D}(\ggg)$ over $\bbf$ to be the maxima of the minimal polynomial degrees for all $D\in \mathcal{D}(\ggg)$ over $\bbf$. By linear algebra (see for example, \cite[Ex.XI.3.8, page 254]{Lang}), $\wp_\bbf(\mathcal{D}(\ggg))=l$. Due to Lemma \ref{lem: central simple}, $\Theta(U(\ggg))$ is a central simple superalgebra over $\bk$.
We can similarly define the minimal polynomial degree $\wp_\bk(\Theta(U(\ggg))$ over $\bk$. By the definition of $\Theta$, it is easily known that $\wp_\bk(\Theta(U(\ggg))\leq \wp_\bbf(\mathcal{D}(\ggg))$. Hence $\sqrt{\dim_\bk  \Theta(U(\ggg))}\leq l$. Thus $\dim_\bbf\mathcal {D}(\ggg)\geq\dim_{\bk}\Theta(U(\ggg))$.

As to  the third
inequality, it is because that $V_\bk$ is a $U_\chi(\ggg)$-module, and
$Z_\chi(\lambda)$ is an irreducible $U_\chi(\ggg)$-module, while all
irreducible $U_\chi(\ggg)$-modules have the same dimension. Hence, it
is enough to prove the second  inequality. For this we only need to
verify that $V_\bk$ is an irreducible $\Theta(U(\ggg))$-module. If this is not true, then $V_\bk$ must contain a proper irreducible submodule $W$ which is of the form $\Theta(U(\ggg))\Theta^V(uu_{1})$. {Put $U'=\cald(\ggg)uu_{1}$.
%Then$W=\Theta(U')$. %\color{red}
Noticing that in $\cald(\ggg)$,  the left
ideal $U'$ must be a direct sum of some simple left ideals (cf. \cite[Proposition 2.11]{Jo}), we may
assume that $U'$ contains the summand  $\cald(\ggg)u_i$ which naturally satisfies $\Theta(u_i)\ne 0$. Therefore, $W$ contains $\Theta(U(\ggg)u_i)$. On the other side, under the action of $\Theta$ we have that  $\Theta(U(\ggg)u_i)$ is isomorphic to
$V_{\bk}=\Theta(U(\ggg)u_{1})$, as $\Theta(U(\ggg))$-modules. This gives rise to a contradiction:  $V_\bk\supsetneqq W\supset \Theta(U(\ggg)u_i)\cong V_\bk$. This completes the proof of (\ref{dim}).}

Recall $(\dim Z_\chi(\lambda))^2=p^{s-r}2^t$. By (\ref{dim}), we have $\dim_{\bbf}\mathcal {D}(\ggg)\geq p^{s-r}2^t$. {Owing to (\ref{eq: free over Zp}),  $\dim_{\Frac(\cz_p)}\mathcal {D}(\ggg)=p^s 2^t$.
Therefore, $\dim_{\Frac(\cz_p)}\mathcal {D}(\ggg)\slash \dim_{\bbf} \mathcal{D}(\ggg)\leq p^r$.  Note that  $\cald(\ggg)$ can be regarded as $\Frac(\cz_p)$-space and $\Frac(\cz)$-space respectively. Both spaces on $\cald(\ggg)$ are denoted by $\cald(\ggg)_{\Frac(\cz_p)}$ and $\cald(\ggg)_{\Frac(\cz)}$ respectively. Then we have
 $\cald(\ggg)_{\Frac(\cz)}=\cald(\ggg)_{\Frac(\cz_p)}
 \otimes_{\Frac(\cz_p)}\Frac(\cz)$. Consequently, {$\dim_{\Frac(\cz_p)}\cald(\ggg)\geq \dim_\bbf \cald(\ggg)\dim_{\Frac(\cz_p)}\bbf$.} Hence
$$[\Frac(\cz):\Frac(\cz_p)]\leq \dim_{\Frac(\cz_p)}\mathcal {D}(\ggg)\slash \dim_{\bbf} \mathcal
{D}(\ggg)\leq p^r.$$}
 In Combination  with Lemma \ref{sep1}, we have  $[\Frac
(\cz):\Frac(\cz_p)]=p^r=[\Frac(\cz):\Frac(\cz_p)]_s$. Hence, the inequalities in (\ref{dim}) turn out to
 be real equalities.
\end{proof}

\subsubsection{}
For  further arguments, we need some preparation involving (geometric) quotient spaces
 (see \cite{BB}, \cite{Bor} and \cite{MFK} for more details).

\begin{defn}\label{quo} (cf. \cite[{\S}II.6.1]{Bor})
Let $\zeta: V\rightarrow W$ be a $\bk$-morphism of $\bk$-varieties.
We say $\zeta$ is a {\em quotient morphism} {if the following items are satisfied:}
\begin{itemize}
\item[(1)]
$\zeta$ is surjective and open.
\item[(2)]
If $U\subset V$ is open, then $\zeta_{*}$ induces an isomorphism from
$\bk[\zeta(U)]$ onto the set of $f\in \bk[U]$ which are constant on the
fibers of $\zeta|_{U}$.
\end{itemize}
\end{defn}

\begin{defn} (cf. \cite[{\S}II.6.3]{Bor})
We fix a  $\bk$-group $H$ acting $\bk$-morphically on a variety $V$.
\begin{itemize}
\item[(1)] An orbit map is a surjective morphism $\zeta: V\rightarrow W$ of varieties such that the fibers of $\zeta$ are the orbits of $H$ in $V$.
\item[(2)]  {\em
A (geometric) quotient  } of $V$ by $H$ over $\bk$ is an orbit map  $\zeta:
V\rightarrow W$ which is a quotient morphism over $\bk$ in the
sense of Definition \ref{quo}. In this case, $W$ is called {\sl the (geometric) quotient space of $V$
by $H$, denoted by $V/H$}.
\end{itemize}
\end{defn}

Now we have a key lemma for our main theorem.
\begin{lemma}\label{sep2}
$[\Frac(\cz)^{G_{\ev}}:
\Frac(\cz_p)^{G_{\ev}}]=p^{r}$.
\end{lemma}
\begin{proof} The proof is a "super" analogue of \cite[Lemma 4.4]{KW}. We still give a full exposition for the readers' convenience.

By Proposition \ref{sepc}, $\Frac(\cz)$ is a separable
extension over $\Frac(\cz_p)$, and the separable degree is
$p^r$. Consider the surjective $G_{\ev}$-equivariant map of varieties
$\phi: \Spec(\cz)\rightarrow \Spec (\cz_p)$ induced
from the algebra embedding $\cz_p\hookrightarrow
\cz$. Recall that $\Spec(\cz_p)$ can be regarded as
$\ggg_{\bar 0}^{*(1)}$. By Rosenlicht's theorem (cf. \cite{Ros}), $\ggg_{\bar
0}^*$ has a $G_{\ev}$-stable open dense subset $V$ such that the quotient space of $V$ by $G_{\ev}$ can be defined. Since
$V_0:=G_{\ev}\cdot \Omega_1$ is an open set of $\ggg_{\bar 0}^*$,
$U=V\cap V_0$ is a $G_{\ev}$-stable open set of irreducible variety
$\ggg_{\bar 0}^*$. Moreover by the property of geometric quotients, we have
$\bk(U/G_{\ev})=\Frac(\cz_p)^{G_{\ev}}$, where the geometric
quotient $U\slash G_{\ev}$ is the $G_{\ev}$-orbit space of $U$. Set
$W:=\phi^{-1}(U)$, a $G_{\ev}$-stable open set of
$\Spec(\cz)$. Note that the stabilizer $(G_\ev)_\chi$ of $\chi$ in  $G_{\ev}$ acts
trivially on $\phi^{-1}(\chi)$, for $\chi\in U$. Hence, the geometric
quotient $W\slash G_{\ev}$ exists, which then induces naturally a morphism of orbit
spaces
$$\phi_{G_{\ev}}: W\slash G_{\ev}\rightarrow U\slash G_{\ev}.$$
Comparing  Theorem \ref{sep}(3), we know that $\phi_{G_{\ev}}$ is still a
separable morphism with degree $p^r$. The proposition follows.
\end{proof}

Recall a fact for ordinary Lie algebra case that (cf. \cite[Lemma 4.6]{KW})
\begin{align}\label{pcen}
\Frac(\cz_p)^{G_{\ev}}=\Frac({\cz_p}^{G_\ev}).
\end{align}

Then we will have the following super version of (\ref{pcen}).

\begin{prop}\label{zfra} $\Frac(\cz)^{G_{\ev}}=\Frac(\cz^{G_{\ev}})$.
\end{prop}

Before the proof, we need the following result.

\renewcommand{\labelenumi}{(\roman{enumi})}
\begin{theorem}(cf. \cite[I.5.3]{BB}) \label{afq}
Let $X$ be an affine algebraic variety with an action of an affine
reductive algebraic group. Then there exists a unique (up to an
isomorphism) {affine} algebraic variety $Y$ such that $\bk[Y]=\bk[X]^{G}$, The
variety $Y$ is called the affine quotient space of $X$ by the action
of $G$. The inclusion $\bk[X]^{G}\subset \bk[X]$ induces a morphism
$\zeta: X\rightarrow Y$. The morphism $\zeta$ is called \sl{the
affine quotient morphism }of $X$ by the action of $G$.
\end{theorem}

\noindent\textbf{Proof of Proposition \ref{zfra}}: We proceed the proof by steps.

(1) By the Hilbert-Nagata Theorem, $\cz^{G_\ev}$ is finitely generated (cf. \cite[Theorem A.1.0]{MFK}).
Set
$X=\Spec(\cz)$, $Y=\Spec(\cz^{G_{\ev}})$, denote by
$\zeta: X\rightarrow Y$, the morphism of irreducible varieties
induced from the inclusion $\cz^{G_{\ev}}\hookrightarrow
\cz$. By  Theorem \ref{afq}, $\zeta$ satisfies the property of
affine quotient map. Then $\zeta $ is  submersive (cf. \cite[Theorem
A.1.1]{MFK} or \cite[Theorem I.5.4]{BB}). Moreover, by
\cite[Proposition 1.9]{MFK}, there exists a maximal $G_{\ev}$-stable
open subset $X^{'}$ of $X$ such that  $\zeta$ restricting on $X^{'}$ is
geometric quotient map.

{ (2) We claim that such $X'$ is nonempty. For getting it, we observe that the points of $X$ are described by $\ker(\Upsilon)$ for the central characters $\Upsilon$ of $U(\ggg)$ which are algebra homomorphisms $\Upsilon: \cz\rightarrow \bk$.  Associated any irreducible $U(\ggg)$-module $L$, one has a unique central character $\Upsilon_L$. Recall that for any given irreducible module $L$, it admits a unique $\chi\in\ggg^*_\bz$ such that $L$ is an irreducible $U_\chi(\ggg)$-module (see \S\ref{reduced}).

Keep the notations as in \S\ref{purelyevenarguments}.
Consider $\chi=0$. The corresponding irreducible modules of $U(\ggg)$ are just restricted irreducible modules which  are parameterized by finite weights $\lambda$ with $\lambda\in \hhh^*$ satisfying $\lambda(H_i)^p=\lambda(H_i)$ for $i=1,\cdots,r$. Such $\lambda$ can be regarded elements in $\bbf_p^r$.
This is to say,
the isomorphism classes of irreducible modules of $U_\chi(\ggg)$ are finite, which are presented by some $\{L_0(\lambda)\mid \lambda\in\bbf_p^r\}$. Here
$L_0(\lambda)$ is a unique irreducible quotient of the baby Verma module $Z_0(\lambda):=U_0(\ggg)\otimes_{U_0(\bbb)}\bk_\lambda$
(see for example,  \cite[\S4.1.2]{PS}).

On the other hand, $\cz$ is $G_\ev$-stable under  $\textsf{Ad}$-action. Consequently, $X$ is a $G_\ev$-scheme. Note that $\textsf{Ad}(G_\ev)$-action preserves the $p$-structure of $\ggg$, i.e. $\textsf{Ad}(g)$ commutes with the $p$-mapping $[p]$ (see \cite[\S I.3.5, \S I.3.19]{Bor} or  \cite{SZheng}). Hence under $\textsf{Ad}(G_\ev)$-action, the isomorphism class of restricted irreducible modules are preserved, i.e. those form a $G_\ev$-stable subset of $\Spec(U(\ggg))$  by the same arguments in  the proof of Lemma \ref{sep1}.  Furthermore,  the $G$-orbit of $x:=\pi(L_0(\lambda_i))$ in $X=\Spec(\cz)$ must be finite as mentioned just above. Consider the structure sheaf  $\scrl$ of $X$ which can be regarded a $G_\ev$-linearized line bundle over $X$, i.e. $G_\ev$-action on $\scrl$ agrees with the $G$-action on $X$.
%The inclusion $\cz^{G_\ev}\hookrightarrow \cz$ gives rise to a dominant %morphism $f: \Spec(\cz)\rightarrow \Spec(\cz^{G_\ev})$.
Obviously there exists a $G$-invariant section $s$ of $\scrl$ arising from the invariant ring $\cz^{G_\ev}$ such that $x\in \text{supp}(s)$. Note that $G_\ev$ is a connected reductive algebraic group, and that the orbit $G_\ev.x$ is finite, consequently, closed in $\text{supp}(s)$.  From  \cite[Theorem 6.1.1(3)]{BB} (or \cite[\S4]{MFK}), it follows that $X'$ is a nonempty open subset of $X$.
}

(3)
Let $Y^{'}=\zeta(X^{'})$. Then
$X^{'}=\zeta^{-1}(Y^{'})$ (cf. \cite[Theorem A.1.1]{MFK} or
\cite[Theorem  I.5.4]{BB}). Hence $Y^{'}$ is an open set of $Y$. From the irreducibility of $X, Y$, it follows that
$\Frac(\bk[X])=\Frac(\bk[X^{'}])$ and $\Frac(\bk[Y])=\Frac(\bk[Y^{'}])$.
Again by \cite[Proposition II.6.5]{Bor}, $\Frac(\bk[X^{'}])^{G_{\ev}}\cong
\Frac(\bk[Y^{'}])$. Therefore, $\Frac(\cz)^{G_{\ev}}=\Frac(\cz^{G_{\ev}})$.
We complete the proof.

\iffalse
\begin{remark} We have a more general result than the above proposition.
\vskip5pt
\bf{$G$-quotient Lemma.}\sl{
Let $\bk$ be a field with prime characteristic, and $A$ be a finitely
generated integral $\bk$-ring. If $G$ is a reductive
algebraic group or finite group of $A$ acting on $A$ as automorphism
of $A$, then $\Frac(A)^{G}=\Frac(A^{G})$.}
\begin{proof}
When $G$ is a reductive algebraic group, the proof can be seen in Proposition \ref{zfra}. For the case of finite group $G$, by Deligne's theorem  (cf. \cite[Theorem I.4.3.2]{BB}), the canonical
affine quotient map $\text{Spec}(A)\rightarrow \text{Spec}(A^G)$ is a
geometric quotient map. The lemma follows.
\end{proof}
\end{remark}
\fi

\iffalse
\begin{theorem}
$$\Frac(\gamma(\cz^{G_{\ev}}))=\Frac(U(\frak h)^{\fw}).$$
%\end{enumerate}
\end{theorem}
\fi

\subsubsection{}\label{sec: end proof}
\noindent \textbf{Proof of Theorem \ref{secondtheorem}}(1): According to \cite[Lemma 5.4]{KW} and its proof,
we know that $U(\frak h)^{\fw}$ is
integral over $(U(\frak h)\cap\cz_p) ^{\fw}$,
$[\Frac(U(\frak h)^{\fw}):\Frac((U(\frak h)\cap\cz_p)
^{\fw})]\leq p^{r}$, and that $\gamma({\cz_p}^{G_\ev})=(U(\frak
h)\cap\cz_p) ^{\fw}$. Taking  Proposition
\ref{inj} into account, we consider the following computation:
\begin{eqnarray*}
&&[\Frac(U(\frak h)^{\fw}): \Frac((U(\frak h)\cap\cz_p)
^{\fw})]\\
&=&[\Frac(U(\frak
h)^{\fw}): \Frac(\gamma({\cz_p}^{G_\ev}))]\\
&=&[\Frac(U(\frak h)^{\fw}): \Frac(\gamma(\cz^{G_{\ev}}))][\Frac(\gamma(\cz^{G_{\ev}})): \Frac(\gamma({\cz_p}^{G_\ev}))]\\
&=&[\Frac(U(\frak h)^{\fw}):\Frac(\gamma(\cz^{G_{\ev}}))][\Frac(\cz^{G_{\ev}}):\Frac({\cz_p}^{G_\ev})].
\end{eqnarray*}
By Lemma \ref{pcen}, Proposition \ref{zfra}, and Lemma
\ref{sep2}, we have
$$[\Frac(\cz^{G_{\ev}}):\Frac({\cz_p}^{G_\ev})]=
[\Frac(\cz)^{G_{\ev}}:\Frac(\cz_p)^{G_{\ev}}]=p^{r}.$$
Hence it follows  that $\Frac(U(\frak
h)^{\fw})=\Frac(\gamma(\cz^{G_{\ev}}))$. Thus we complete the proof.

\begin{remark} We remind the readers to compare the result with the one for the complex Lie superalgebras (cf. \cite[Lemmas 1-2]{Kac3}).
\end{remark}

\vskip0.5cm
\noindent \textbf{Proof of Theorem \ref{secondtheorem}}(2): Consider the inclusion
$\cz_p\hookrightarrow\widetilde\cz\hookrightarrow\cz$,
and $\cz^{G_{\ev}}\hookrightarrow\widetilde\cz
\hookrightarrow\cz$. There is a commutative diagram of
dominant morphism of irreducible varieties
$$
\begin{array}{c}
\xymatrix{
&&\text{Spec}(\cz_p)\\
\text{Spec}(\cz)\ar[r]^{\varphi_{1}}\ar[drr]_{f_{1}}\ar[urr]^{f_{3}}
&\text{Spec}(\widetilde\cz)\ar[ur]_{\varphi_{2}}
\ar[dr]^{f_{2}}&\\
&&\text{Spec}(\cz^{G_{\ev}}). }
\end{array}
$$
By the proof of Proposition \ref{zfra}, there exist a
$G_{\ev}$-stable nonempty open subset $V$ of $\Spec(\cz)$, and a
$G_{\ev}$-stable nonempty open subset $W$ of
$\text{Spec}(\cz^{G_{\ev}})$ such that $f_{1}|_{V}:
V\rightarrow W$ is a geometric quotient map. So, for a $w\in W$, all elements of
$f_{1}^{-1}(w)$ fall in the same $G_{\ev}$-orbit. Recall that $G_{\ev}\cdot
\Omega_{1}$ contains a $G_{\ev}$-stable   open  dense subset of
$\text{Spec}(\cz_p)$, say $U_1$.  {Set $U=\varphi_2^{-1}(U_1)\cap f_{2}^{-1}(W)$.} By the irreducibility of
$\text{Spec}(\widetilde\cz)$ and of $\Spec(\cz^{G_{\ev}})$, and by the dominance of $f_{2}$, we have that $U$ is a
$G_{\ev}$-stable  open  dense subset of $\text{Spec}(\widetilde\cz)$.

Take $x\in U$. By Proposition \ref{cev}$(2)$ and Proposition
\ref{sepc}, both $\varphi_{1}$ and $\varphi_{2}$ are finite separable
morphisms. We assume that  $\varphi^{-1}_{1}(x)=\{x_{1},\cdots,x_{d}\}$. We assert that $d=1$. Actually, by the choice of $x$ there
exists $w\in W$ such that $f_{2}(x)=w$. Then
$\varphi^{-1}_{1}(x)\subseteq\varphi^{-1}_{1}(f_{2}^{-1}(w))\subseteq
(f_{2}\circ \varphi_{1})^{-1}(w)=f_{1}^{-1}(w)$. According to the
above arguments and the property of geometric quotient,
$x_{1},\cdots,x_{t}$ are contained in one orbit. Now
$\cz^{G_{\ev}}\subset \widetilde\cz$, hence the stabilizer $(G_{\ev})_{x}$ of $x$ in $G_\ev$ permutes them transitively.
Set $y=\varphi_{2}(x)$. Then $y\in G_{\ev}\cdot \Omega_{1}$, and
$(G_{\ev})_{x}\subseteq(G_{\ev})_{y}$.  As  $\textbf{Z}_\fw(\hhh^*)=\textbf{Z}_\fw(\hhh)=1=\fw_\chi$ for $\chi\in \Omega_1$ (see \S\ref{Weylgroups}), %Equation (\ref{stablizer}),
$(G_{\ev})_{y}$ acts trivially on the fiber over $y$ of
$(\varphi_{2}\varphi_{1})^{-1}$,  the same is true for
$(G_{\ev})_{x}$. Hence $d=1$ as we asserted. Furthermore, since a
geometric quotient map is  submersive, two open sets
$\varphi_1^{-1}(U)$ and $U$ are isomorphic. Therefore, we get that
$\Spec(\cz)$ and $\Spec(\widetilde\cz)$ are birational
equivalent. The proof  is completed.

\section{Birational equivalence of the Zassenhaus varieties and their rationality}\label{sec: Slodowy slices}

Keep the notations and assumptions as before. In particular, $\ggg=\ggg_\bz\oplus\ggg_\bo$  is a basic classical Lie superalgebras as list in \S\ref{tab: basic cla}. The subalgebra $\widetilde{\cz}$ of $\cz=\text{Cent}(U(\ggg))$ is generated  the $p$-center  $\cz_p$ and the Harish-Chandra center $\cz_\hc=\cz^{G_\ev}$.

Besides the assumption on $p$ as in \S\ref{tab: basic cla}, we assume that  the purely-even reductive group $G_\bz$  satisfies the following standard hypotheses on $\bG$.

\begin{hypothesis}\label{hyp: standard}  (see \cite{Jan}, \cite{BGo}) The following hypotheses are standard for an connected reductive algebraic group $\bG$ and its Lie algebra $\g$ with $\Lie(\bG)=\g$.
\begin{itemize}
\item[(A)] the derived group $D\bG$ of $\bG$ is simply-connected;
\item[(B)] $p$ a good prime for $\bG$;
\item[(C)] $\g$ has a non-degenerate $\bG$-invariant bilinear form.
\end{itemize}
\end{hypothesis}

\subsection{} Recall $\ggg_\bz$ is a reductive Lie algebra. The center $\frakZ$ of $U(\ggg_\bz)$ can be described as below.

\begin{theorem} (Veldkamp's theorem, \cite{Ve}, \cite{KW}, \cite{MR}, \cite{BGo}, {\sl etc.})\label{thm: Vel}
Let $\frakZ_p$ denote the $p$-center of $U(\ggg_\bz)$ and $\frakZ_\hc$ denote the Harish-Chandra center of $U(\ggg_\bz)$, i.e.  $\frakZ_\hc=U(\ggg_\bz)^{G_\ev}$. Then $\frakZ$ is generated by $\frakZ_p$ and $\frakZ_\hc$. More precisely,
$$\frakZ\cong \frakZ_p\otimes_{\frakZ_p^{G_\ev}}\frakZ_\hc.$$
\end{theorem}

Keep Veldkamp's theorem in mind. We have the following result.

\begin{prop}\label{prop: 3.2} There is an algebra isomorphism
$$\Frac(\widetilde\cz)\cong \Frac(\frakZ).$$
\end{prop}

\begin{proof} By definition, $\cz_p$ actually coincides with $\frakZ_p$.  We focus the concern on $\cz_\hc$ and $\frakZ_\hc$. Owing to Proposition \ref{inj}, $\cz_\hc$ can be regarded a subalgebra of $U(\hhh)^\fw$ under the injective homomorphism $\gamma$ which arises from the  Harish-Chandra homomorphism in the super case. Note that the classical Harish-Chandra theorem shows that there is an isomorphism from $\mathcal{HC}: U(\ggg_\bz)^{G_\ev}$ onto $U(\hhh)^\fw$ (see for example, \cite[\S9]{Jan}). So there exists an injective homomorphism $\vartheta$  from $\cz_\hc$ into $\frakZ_\hc$, which is exactly $\mathcal{HC}^{-1}\circ \gamma$. So $\cz_\hc$ can be regarded a subalgebra of $\frakZ_\hc$, under $\vartheta$.

Note that $\widetilde{\cz}$ is generated by $\cz_p$ and
$\cz_\hc$. Keep in mind that $\cz_p=\frakZ_p$.  Then $\cz_p\cap \cz_\hc=\cz_p^{G_\ev}$ which coincides with
$\frakZ_p^{G_\ev}$. Therefore, by Theorem \ref{thm: Vel} we can regard
$\widetilde\cz$ as a subalgebra of $\frakZ$ which is generated by
$\frakZ_p$ and $\vartheta(\cz_\hc)$.
So $\Frac(\widetilde\cz)$ is
isomorphic to the subfield of $\Frac(\frakZ)$ generated by $\Frac(\cz_p)$ and
$\Frac(\vartheta(\cz_\hc))$.
However, under $\gamma$ the fraction $\Frac(\cz_\hc)$ is isomorphic to the fraction of $U(\hhh)^\fw$ (Theorem \ref{secondtheorem}(1)),  which is exactly isomorphic to
$\Frac(\frakZ_\hc)$. Hence $\Frac(\widetilde\cz)$ is isomorphic to  the subfield of $\Frac(\frakZ)$ generated by $\Frac(\frakZ_p)$ and
$\Frac(\frakZ_\hc)$. This subfield exactly coincides with
$\Frac(\frakZ)$.  Hence  $\Frac(\widetilde\cz)$ is isomorphic to
$\Frac(\frakZ)$. The proof is completed.
\end{proof}

By combination of  the above proposition and Theorem \ref{secondtheorem}, we have

\begin{theorem}\label{thm: two cent birational eq} There is an isomorphism of  fractions
 $$\Frac(\cz)\cong \Frac(\frakZ).$$
\end{theorem}

\subsection{}
Set $\bZ=\Spec(\cz)$ and $\bzw=\Spec(\widetilde\cz)$. The affine scheme $\bZ$ is called the Zassenhaus variety for $\ggg$ as usual. By Theorem \ref{secondtheorem}, the Zassenhaus variety $\Spec(\cz)$ for $\ggg$ is birationally equivalent to $\bzw$. Due to Proposition \ref{prop: 3.2}, the Zassenhaus variety $\Spec(\frakZ)$ is birationally equivalent to $\bzw$. So we can regard that $\Spec(\cz)$ and $\Spec(\frakZ)$ are birationally equivalent.

\subsection{} Under Hypothesis \ref{hyp: standard}, Tange shows that
$\Frac(\frakZ)$ is rational for $\frakZ=\text(Cent)(U(\g))$ with $\g=\Lie(\bG)$ being a reductive Lie algebra (see \cite{T}). Hence we have the following corollary to Theorem \ref{thm: two cent birational eq}.

\begin{corollary} Suppose $\ggg=\ggg_\bz\oplus\ggg_\bo$ is a basic classical Lie superalgebra as the list in \S\ref{tab: basic cla}, and $\ggg_\bz=\Lie(G_\ev)$ satisfies Hypothesis \ref{hyp: standard}. Then $\bZ$ is rational, i.e. $\Frac(\cz)$ is purely transcendental over $\bk$.
\end{corollary}

\subsection*{Acknowledgement} The authors are grateful to the referee for his/her helpful comments and suggestions which enable them to improve and revise the manuscript.

\end{document}